\documentclass[12pt,letterpaper,twoside,english]{amsart}
\usepackage[T1]{fontenc}
\usepackage[latin9]{inputenc}
\usepackage{babel}
\usepackage{mathrsfs}
\usepackage{amsthm}
\usepackage{amstext}
\usepackage{amssymb}
\usepackage{graphicx}
\usepackage{setspace}
\usepackage{esint}
\usepackage{comment}
\usepackage{enumerate}
\usepackage{amscd,color}
\usepackage[unicode=true,pdfusetitle,
 bookmarks=true,bookmarksnumbered=false,bookmarksopen=false,
 breaklinks=false,pdfborder={0 0 1},backref=false,colorlinks=false]
 {hyperref}
\makeatletter
\pdfpageheight\paperheight
\pdfpagewidth\paperwidth

\numberwithin{equation}{section}
\numberwithin{figure}{section}
\theoremstyle{plain}
\newtheorem{thm}{\protect\theoremname}[section]
  \theoremstyle{definition}
  \newtheorem{example}[thm]{\protect\examplename}
  \theoremstyle{remark}
  \newtheorem{rem}[thm]{\protect\remarkname}
  \theoremstyle{plain}
  \newtheorem{prop}[thm]{\protect\propositionname}
  \theoremstyle{plain}
  \newtheorem{cor}[thm]{\protect\corollaryname}
  \theoremstyle{definition}
  \newtheorem{defn}[thm]{\protect\definitionname}
  \theoremstyle{plain}

\newcommand{\h}{\hbox}
\newcommand{\q}{\quad}

\newcommand{\ms}{\par\medskip}
\newcommand{\sk}{\par\smallskip}
\newcommand{\bsn}{\par\bigskip\noindent}
\newcommand{\msn}{\par\medskip\noindent}
\newcommand{\skn}{\par\smallskip\noindent}
\newcommand{\ges}{\geqslant}
\newcommand{\les}{\leqslant}
\newcommand{\mopl}{\hbox{$\bigoplus$}}

\newcommand{\Dc}{{\mathcal D}}
\newcommand{\Hc}{{\mathcal H}}
\newcommand{\Ic}{{\mathcal I}}
\newcommand{\Oc}{{\mathcal O}}
\newcommand{\Xt}{{}\,\widetilde{\!X}{}}
\newcommand{\Db}{{\mathbb D}}
\newcommand{\Rb}{{\mathbf R}}
\newcommand{\Qb}{{\mathbb Q}}
\newcommand{\Cb}{{\mathbb C}}
\newcommand{\Zb}{{\mathbb Z}}
\newcommand{\Ks}{K^{\ssb}}
\newcommand{\Kps}{K'{}^{\ssb}}
\newcommand{\dd}{\partial}
\newcommand{\Gr}{{\rm Gr}}
\newcommand{\De}{\Delta}
\newcommand{\bl}{\bigl}
\newcommand{\br}{\bigr}
\newcommand{\ssb}{\raise.15ex\h{${\scriptscriptstyle\bullet}$}}
\newcommand{\ssc}{\,\raise.15ex\hbox{${\scriptstyle\circ}$}\,}
\newcommand{\into}{\hookrightarrow}
\newcommand{\simto}{\,\,\rlap{\hskip1.3mm\raise1.4mm\hbox{$\sim$}}\hbox{$\longrightarrow$}\,\,}

\usepackage{enumitem}
\usepackage{amsmath}
\usepackage{amsfonts}
\usepackage{amssymb}
\usepackage[all]{xy}
\usepackage{array}

\def\RR{\mathbb{R}}
\def\CC{\mathbb{C}}
\def\QQ{\mathbb{Q}}
\def\PP{\mathbb{P}}
\def\ZZ{\mathbb{Z}}
\def\HH{\mathbb{H}}
\def\VV{\mathbb{V}}
\def\LL{\mathbb{L}}
\def\DD{\mathbb{D}}
\def\str{\mathsf{S}}
\def\aff{\mathbb{A}}

\def\fx{\mathfrak{X}}
\def\cE{\mathcal{E}}
\def\cx{\mathcal{X}}
\def\cs{\mathcal{S}}
\def\ch{\mathcal{H}}
\def\ck{\mathcal{K}^{\bullet}}
\def\cy{\mathcal{Y}}
\def\cl{\mathcal{L}}
\def\ci{\mathcal{I}}

\def\cp{\mathcal{P}^{\bullet}}
\def\cn{\mathcal{N}}
\def\co{\mathcal{O}}
\def\cv{\mathcal{V}}

\def\fsm{f^{\mathrm{s}}}
\def\ppsi{{}^p \psi}
\def\pphi{{}^p \phi}
\def\Xs{X_{\sigma}}

\def\image{\mathrm{im}}
\def\an{\mathrm{an}}
\def\nc{\mathrm{nc}}
\def\gr{\mathrm{Gr}}
\def\sing{\mathrm{sing}}
\def\var{\widetilde{\mathtt{var}}}
\def\can{\mathtt{can}}
\def\gy{\mathtt{gy}}
\def\sp{\mathtt{sp}}

\def\pha{\mathrm{ph}}
\def\van{\mathrm{van}}
\def\lm{\mathrm{lim}}

\def\IH{\mathrm{IH}}
\def\IC{\mathrm{IC}^{\bullet}}
\def\sm{\mathrm{sm}}
\def\cone{\mathrm{Cone}}

\def\pr{{}^p R}
\def\pl{{}^p \cl}
\def\ph{{}^p \ch}
\def\ve{\varepsilon}
\def\eb{\epsilon^{\bullet}}
\def\ox{\co_X}
\def\xb{X^{\bullet}}
\def\ob{\co_{\xb}}
\def\uw{\underline{w}}

\def\uo{\underline{0}}

\newcommand{\vinj}{\rotatebox[origin=c]{-90}{$\hookrightarrow$}}

\makeatother

  \providecommand{\corollaryname}{Corollary}
  \providecommand{\definitionname}{Definition}
  \providecommand{\examplename}{Example}
  \providecommand{\lemmaname}{Lemma}
  \providecommand{\propositionname}{Proposition}
  \providecommand{\remarkname}{Remark}
\providecommand{\theoremname}{Theorem}

\makeatletter
\let\@wraptoccontribs\wraptoccontribs
\makeatother

\begin{document}

\title[Hodge Theory of Degenerations, I]{Hodge Theory of Degenerations, (I):\\Consequences of the Decomposition Theorem}

\author{Matt Kerr}
\address[Matt Kerr]{Washington University in St. Louis, Department of Mathematics and Statistics, St. Louis, MO 63130-4899, USA}
\email{matkerr@math.wustl.edu}

 \author{Radu Laza}
 \address[Radu Laza]{Stony Brook University, Department of Mathematics, Stony Brook, NY 11794-3651,USA}
\email{rlaza@math.stonybrook.edu}

\contrib[with an appendix by]{Morihiko Saito}
 \address[Morihiko Saito]{RIMS Kyoto University, Kyoto 606-8502, Japan}
  \email{msaito@kurims.kyoto-u.ac.jp} 
\bibliographystyle{amsalpha}

\thanks{During its writing, MK was supported by NSF Grant DMS-1361147; RL was supported by NSF Grant DMS-1802128 and DMS-1254812.}
\maketitle

\begin{abstract}
We use the Decomposition Theorem to derive several generalizations of the Clemens--Schmid sequence, relating asymptotic Hodge theory of a degeneration to the mixed Hodge theory of its singular fiber(s).
\end{abstract}

\section{Introduction}

This paper initiates a series of articles on the relationship between the asymptotic Hodge theory of a degeneration and the mixed Hodge theory of its singular fiber(s), motivated by the study of compactifications of moduli spaces.  In this first installment, we concentrate on what may be derived from the Decomposition Theorem (DT) of \cite{BBD} in the setting of mixed Hodge modules \cite{toh}, including several variants of the 
Clemens--Schmid (C-S) exact sequence \cite{Cl} (also see \cite{GN}, \cite{MR1602375}) and basic results on the vanishing cohomology.  In a forthcoming sequel \cite{KL}, referred to henceforth as \textbf{Part II}, we investigate the vanishing cohomology in further detail, and give several applications to geometric degenerations.

The period map is the main tool for studying the moduli spaces of abelian varieties, $K3$ surfaces \cite{Sh80,Lo03}, and related objects such as hyper-K\"ahler manifolds \cite{Huy11,KLSV} and cubic 3-folds and 4-folds \cite{Voi86,LS07,La10,ACT11}.  What these ``classical'' examples have in common is a ``strong global Torelli'' property, to the effect that the period map embeds each moduli space as an open subset of a locally symmetric variety.  This facilitates the comparison, or even an explicit birational correspondence, between Hodge-theoretic (i.e. toroidal \cite{AMRT} and Baily-Borel) compactifications and geometric ones (such as KSBA or GIT compactifications); see for example the series of papers \cite{La16,log1,log2,log3}.  A program led by Griffiths, with contributions of many people (including the authors), aims to extend the use of period maps in studying moduli to the ``non-classical'' case, especially surfaces of general type with $p_g\ge 2$ and Calabi-Yau threefolds, with the premise that this strong connection between compactifications should remain. 
In particular, the geometric boundary (suitably blown up) carries variations of limiting mixed Hodge structures (LMHS) on its strata, which in principle yield period maps to Hodge-theoretic boundary components\footnote{Suitable compactifications are known for locally symmetric subvarieties of Hodge-theoretic classifying spaces, and are expected to exist (in horizontal directions) in general.}.  The challenge is thus to compute these LMHS, and their associated monodromies, as well as possible from the geometry of the (singular) fibers over the geometric boundary.

There are two main parts to this challenge. {\it The first is to compute the MHS on the singular fibers and relate this to the invariant cycles in the LMHS.}  For the ideal topological set-up, that of a semistable degeneration $\cx \overset{f}{\to} \Delta$ over a disk (centered at the origin) with singular fiber $X_0$, a piece of the Clemens--Schmid sequence says that 
\begin{equation}\label{eqI1}
H^k_{X_0}(\cx)\overset{\mu}{\to} H^k(X_0) \overset{\sp}{\to} H^k_{\text{inv}}(X_t) \to 0
\end{equation}
is an exact sequence of MHS, with $\image({\mu})$ pure of weight $k$ (and level $\leq k-2$). While this is a very strong statement, the natural degenerations occurring (say) in
GIT (Geometric Invariant Theory \cite{GIT}) or KSBA (Koll\'ar--Shepherd-Barron--Alexeev \cite{ksb}, \cite{alexeev}) compactifications are rarely semistable, and difficult to put in this form via semistable reduction.  Indeed, the philosophy of the minimal model program (MMP) is that, for sufficiently ``mild'' singularities on $\cx$ and $X_0$, we need \emph{not} carry out semistable reduction, as illustrated by papers from \cite{shahins,Sh80,Sh81} to \cite{La10,KLSV}.  

In accord with this principle, we have largely focused this paper on various generalizations of Clemens--Schmid, starting with the simple observation (cf. Theorem \ref{thMain} and \eqref{eq1.5b}) that \eqref{eqI1} remains valid for smooth $\cx$ and projective $f$, regardless of unipotency of monodromy or singularities of $X_0$. Specifically, we have:
\begin{thm}\label{thMain0}
Let $f \colon \cx\to \Delta$ be a flat projective family of varieties over the disk, which is the restriction of an algebraic family over a curve, such that $f$ smooth over $\Delta^*$.  If $\cx$ is smooth, then we have exact sequences of MHS 
\begin{multline}
0\to H^{k-2}_{\lm}(X_t)_{T}(-1)\overset{\sp^{\vee}}{\to}H_{2n-k+2}(X_0)(-n-1)\\ \overset{\gy}{\to}H^k(X_0)\overset{\sp}{\to}H^k_{\lm}(X_t)^{T}\to 0
\end{multline}
for every $k\in\ZZ$, where the outer terms are the coinvariants resp. invariants of the monodromy operator $T$ on the LMHS.
\end{thm}

We are interested especially in versions of Clemens--Schmid for $1$-parameter families arising in the study of KSBA compactifications.  In this direction, we obtain the following result (cf. Thm. \ref{prop1.8b} and Cor. \ref{cor1.8a}), which in particular gives that the frontier Hodge numbers (i.e. $h^{p,q}$ with $p\cdot q=0$) are preserved for such degenerations. Weaker versions  of our result (cf. \cite{shahins}, \cite{steenbrink}) proved to be very useful for the study of degenerations of $K3$ surfaces and hyper-K\"ahler manifolds (e.g. \cite{Sh80}, \cite{log2,KLSV}). 

\begin{thm}\label{thm-intro1} Let $f:\cx\to \Delta$ be as in the first sentence of Theorem \ref{thMain0} (in particular, $\cx\setminus X_0$ is smooth). Suppose that $\cx$ is normal and $\QQ$-Gorenstein, and that the special fiber $X_0$ is reduced.
\begin{itemize}
\item[(i)] If $X_{0}$ is semi-log-canonical (slc), then 
$$\gr_F^0 H^k(X_0)\cong \gr_F^0 H_{\lm}^k (X_t)\cong \gr_F^0 H^k_{\lm} (X_t)^{T^{\text{ss}}} \;\;\;(\forall k),$$ 
where $T=T^{\text{un}}T^{\text{ss}}$ is the Jordan decomposition of the monodromy into unipotent and (finite) semisimple parts.
\item[(ii)] If $X_{0}$ is log-terminal, then additionally $(\forall k)$
$$W_{k-1}\gr_{F}^{0}H_{\lm}^{k}(X_{t})=\{0\}.$$
\end{itemize}
\end{thm}
\begin{rem}
Under the assumption of $X_0$ having du Bois singularities, the first isomorphism of item (i) above is due to Steenbrink \cite{steenbrink}. Koll\'ar--Kov\'acs \cite{KK} (see also \cite[\S6.2]{kol2}) proved that slc singularities are du Bois, recovering our version above.
\end{rem}

Under stronger assumptions (especially smoothness for the total space $\cx$), we are able to go deeper into the Hodge filtration (Theorem \ref{cor1.8b}). We expect that this result (which to our knowledge is new) will play an important role in the study of degenerations of Calabi-Yau threefolds with canonical singularities, and respectively surfaces of general type with $p_g=2$. (Several related questions about these two geometric cases are currently under investigation by the authors and their collaborators under the aegis of Griffiths's program.) 
\begin{thm}[{= Theorem \ref{cor1.8b}}]\label{thm-intro2}  Let $f:\cx\to \Delta$ be as in Theorem \ref{thm-intro1}. Assume that the total space $\cx$ is smooth and the special fiber $X_0$ is log-terminal (or more generally,
has rational singularities). Then 
$$\gr_{F}^{1}H^{k}(X_{0})\cong\gr_{F}^{1}(H_{\lm}^{k}(X_{t}))^{T^{\text{ss}}}.$$\end{thm}

\begin{rem}
The general philosophy of Theorems \ref{thm-intro1} and \ref{thm-intro2} is that the milder the singularities, the closer the relationship between the Hodge structure on the central fiber $X_0$ and the limit Hodge structure is.  In \textbf{Part II} of our paper, we will give some further versions based on the concept of $k$-log-canonicity of Mustata--Popa \cite{MP} (see also \cite{JKYS} for some more recent developments). In the opposite direction, one can ask what happens if $X_0$ is not log canonical. This leads to questions on the Hodge structure of the ``tail''  (e.g. see \cite{hassett} and \cite[Sect. 6]{log2}) occurring in a KSBA stable replacement. While some examples are discussed here, we revisit the topic in a more systematic way in \textbf{Part II}. 
\end{rem}

\begin{rem}
Versions of Theorems \ref{thm-intro1} and \ref{thm-intro2}  (under somewhat weaker, but less geometric assumptions) are the subject of a forthcoming paper joint with M. Saito \cite{KLS}.
\end{rem}

For singular total spaces, there are ``clean'' versions of Clemens--Schmid only for semisimple perverse sheaves \eqref{eq1.4b} (including intersection cohomology \eqref{eq1.4f}).  For us, the importance of semisimplicity with respect to the perverse t-structure was driven home by \cite{BC18}, and we explain in Example \ref{ex1.6a} how this typically fails for $\QQ_{\cx}[d_{\cx}]$ even when it is perverse.  So the versions for usual cohomology with $\cx$ singular are necessarily more partial, as seen in the context of base-change and log-resolutions \eqref{eq1.7d}, quotient singularities \eqref{eq1.7h}, and MMP-type singularities (results in $\S$\ref{S1.8}).  Finally, in Theorem \ref{thmS9} we arrive at an analogue of Clemens--Schmid for the simplest kinds of multiparameter degenerations (smooth total space, snc discriminant divisor), including for instance those termed semistable by \cite{AK}.

{\it The second main aspect to determining the LMHS of a 1-parameter degeneration (without applying semistable reduction) is to tease out of the geometry of $X_0$ those aspects which are invisible to $H^*(X_0)$}.  Here the main tool (for $\cx\to\Delta$) is the exact sequence
\begin{equation}\label{eqI2}
\cdots \to H^k(X_0) \to H^k_{\lm}(X_t) \to H^k_{\van}(X_t) \to H^{k+1}(X_0)\to\cdots
\end{equation}
where $H^k_{\van}(X_t)$ denotes $\HH^{k+d_{\cx}-1}$ of the vanishing cycle sheaf $\pphi_f\,\QQ_{\cx}[d_{\cx}]$ on $X_0$, promoted to a MHS by Saito's realization of $\pphi_f \QQ_{\cx}[d_{\cx}]$ as a mixed Hodge module (MHM) in \cite{mhm}.  We shall refer to \eqref{eqI2} as the \emph{vanishing cycle sequence}. 
Basic results on the \emph{vanishing cohomology} $H^k_{\van}(X_t)$ are proved in Propositions \ref{prop1.4} and \ref{prop1.5a} and Theorem \ref{prop1.5b} here; for instance, in the case of an isolated singularity, its underlying $\QQ$-vector space is the reduced cohomology of the Milnor fiber \cite{Mi}.  These are but a small taste of what will be the main topic in \textbf{Part II} of our study, in which tools such as mixed spectra and the motivic Milnor fiber are used to compute $H_{\van}^*$ for various singularities arising in GIT and MMP.

Of course, there is a vast literature on the subject of relating the cohomology and singularity theory of $X_0$ with the limit cohomology (e.g. \cite{Cl}, \cite{steenbrinkvan},  \cite{shahins}, \cite{kulikov}, \cite{toh}, \cite{dCM2}, \cite{DL}, \cite{KK}, \cite{DS}).  Our purpose in this series is to survey, adapt, and (where possible) improve this for degenerations that occur naturally in the geometric context.  Beyond relating geometric and Hodge-theoretic compactifications of moduli, we anticipate applications to the classification of singularities and KSBA (or semistable) replacements of singular fibers occurring in GIT, as well as to limits of normal functions in the general context of \cite{7K}.

\subsection*{Structure of the paper} In Sections \ref{sec-motivation} and \ref{sec-leray}, we start with a review of the Decomposition Theorem and make some preliminary considerations for our situation. The following three sections discuss the case of the Decomposition Theorem over a curve (with an eye towards one-parameter degenerations). First, in Section \ref{S1.3}, we introduce the vanishing and nearby cycles, and the {\it vanishing cycle triangle} relating them (see \eqref{eq1.3b}), followed by general forms of vanishing, limiting, and ``phantom'' cohomology.  These preliminaries allow us to begin Section \ref{SDT2} with a very general form of the Clemens--Schmid exact sequence, which is eventually specialized to the more recognizable form (Theorem \ref{thMain}) under the assumption of smooth total space. The fact that there is a close connection between the Clemens--Schmid sequence and the Decomposition Theorem is well known to experts (e.g. Remark \ref{rem5n}(ii)). In an Appendix to our paper, M. Saito proves a general (suitable) equivalence between Clemens--Schmid sequence, the local invariant cycle Theorem, and the Decomposition Theorem over a curve.

Some concrete geometric examples are then discussed in Section \ref{sec-examples}. These range from the very classical, e.g. families of elliptic curves with various types of Kodaira fibers (Ex. \ref{ex1.5a}), to examples (Ex. \ref{ex1.5c}) that we encountered in the study of degenerations of $K3$ surfaces (see especially \cite{log2}), to the more exotic example of Katz \cite{Katz} of a family of surfaces ($p_g=2$) with $G_2$ monodromy (whose treatment uses the most general form of Clemens-Schmid). These examples serve both to illustrate the C-S and vanishing-cycle exact sequences in well-known settings, and to show the efficacy of the methods developed here in some less familiar situations. While our examples are not new per se, we believe the discussion of Section \ref{sec-examples} gives a deeper and more conceptual understanding of them. In \textbf{Part II}, further tools are developed, which will allow us to give further examples and applications. 

While there is a suitable general theory (and we only touch on MHM work of Saito), the focus of our paper is on specializing these results to concrete situations relevant for geometric questions (esp. compactifications). We start this discussion with the case of isolated singularities (Section \ref{sec-isolated}), and their relationship to the failure of the DT for non-semisimple perverse sheaves.  We then discuss (Section \ref{sec-cyclic}) another common geometric scenario - that of finite base changes and quotient singularities. While some of the discussion here might seem very special from the perspective of the general theory, in concrete geometric situations (including those considered in \textbf{Part II}) subtle issues arise. We hope that our discussion clarifies some of those issues, and we expect that further applications will be obtained in the future. Some examples (including some that we encountered in our previous work) are included along the way. 

The most novel aspects of our work occur in the last two sections. First, in Section \ref{S1.8}, we discuss the situation of one parameter degenerations of KSBA type. Among other things, we obtain Theorems \ref{thm-intro1} and \ref{thm-intro2} discussed above. In the final Section \ref{S1.9}, we start a discussion of the Hodge theoretic behavior of degenerations over multi-dimensional bases. To our knowledge, very little in this direction exists in the current literature.  We expect that the study of degenerations over multi-dimensional bases will play a more prominent role in the future - especially due to the fast progress on multi-dimensional semistable reduction theorems (Abramovich, Temkin and others, e.g. \cite{ALT}, improving on Abramovich--Karu \cite{AK}). A concrete geometric example where multi-dimensional bases occur and our methods might be relevant, we mention the case of cubic threefolds. In \cite{cubics}, a study of the degenerations of intermediate Jacobians in the classical set-up of normal crossing discriminants is done; while in \cite{LSV}, it is essential to study the degenerations of intermediate Jacobians without blowing up the discriminant to normal crossings. The method used for both of these studies is reduction to curves via Mumford's Prym construction for the intermediate Jacobian. It would be interesting to study the degenerations of intermediate Jacobians directly in terms of cubics (see \cite{brosnan} for a step in this direction). 

\subsection*{Acknowledgments}
The authors wish to thank P. Brosnan, P. Gallardo, and L. Migliorini for correspondence closely related to this work, and the IAS for providing the environment in which, some years ago, this series of papers was first conceived.  We also thank M. Green, P. Griffiths, G. Pearlstein, and C. Robles for our fruitful discussions and collaborations on period maps and moduli during the course of the NSF FRG project ``Hodge theory, moduli, and representation theory''.  Finally, we are grateful to M. Saito for his careful reading and numerous remarks which led to improvements in the exposition. We also thank M. Saito for writing an Appendix to our paper that explains the connection between C-S and DT in full generality.

\section{Motivation: {\it Why the Decomposition Theorem?}}\label{sec-motivation}

For any projective map $f:\cx\to\cs$ of quasi-projective varieties
over $\CC$ and $\ck\in D_{c}^{b}(\cx^{\an})$, the equality of functors%
\footnote{As we continue to work in the analytic topology, the superscript ``an''
will be dropped for brevity.%
} 
\[
R\Gamma_{\cs}\circ Rf_{*}=R\Gamma_{\cx}:\, D_{c}^{b}(\cx^{\an})\to D_{c}^{b}(\cs^{\an})
\]
produces the \emph{Leray spectral sequence} \begin{equation}\label{eq1.1*}
E_{2}^{p,q}=\HH^{p}(\cs,R^{q}f_{*}\ck)\;\implies\;\HH^{p+q}(\cx,\ck).
\end{equation}The accompanying \emph{Leray filtration} \begin{equation}\label{eq1.1a}
\cl^{\rho}\HH^k(\cx,\ck)\;:=\; \image\left\{\HH^k(\cs,\tau_{\leq k-{\rho}} Rf_* \ck)\to \HH^k(\cs,Rf_* \ck)\right\}
\end{equation}(with $\gr_{\cl}^{\rho}\cong E_{\infty}^{\rho,k-\rho}$) may be described
in terms of kernels of restrictions to (special) subvarieties of $\cs$
\cite{Ar}. Hence when $\ck$ has the structure of a MHM, $\cl^{\bullet}$
is a filtration by sub-MHS.

However, we would prefer to have more than just a filtration. Recall
the following classical result of Deligne \cite{D1}:
\begin{thm}
\label{th1.1} If $\cx$ and $\cs$ are smooth, $f$ is smooth projective
(of relative dimension $n$), and $\ck=\QQ_{\cx}$, then \eqref{eq1.1*}
degenerates at $E_{2}$.\end{thm}
\begin{proof} See \cite[Prop. 1.33]{PS}.
\end{proof}
As an immediate corollary, this produces a noncanonical decomposition
\begin{equation}\label{eq1.1b}
H^k(\cx,\QQ)\underset{\nc}{\cong} \oplus_p H^p(\cs,R^{k-p}f_*\QQ)
\end{equation}into MHS, which includes an easy case of the global invariant cycle
theorem. Neither the description of the graded pieces of $\cl^{\bullet}$
nor its splitting in \eqref{eq1.1b} may be valid when $\cx$, $\cs$,
or $f$ is not smooth.
\begin{example}
\label{ex1.1}Let $\cy\to\PP^{1}$ be an extremal (smooth, minimal)
rational elliptic surface with (zero-)section $\sigma$. By Noether's
formula,
\[
h_{\cy}^{1,1}=10+10h_{\cy}^{2,0}-8h_{\cy}^{1,0}-K_{\cy}^{2}=10\,;
\]
and we let $\hat{\cy}\overset{\beta}{\to}\cy$ be the blow-up at a
nontorsion point $p$ on a smooth fiber $Y_{0}$, with $P=\beta^{-1}(p)\cong\PP^{1}$.
Contracting the proper transform $\hat{Y}_{0}$ of $Y_{0}$ yields
an elliptic surface $\cx\overset{f}{\to}\PP^{1}$ with isolated
$\tilde{E}_{8}$ singularity $q\in X_{0}(\cong P)=f^{-1}(0)$, since
$\hat{Y}_{0}^{2}=-1$.

First consider the Leray spectral sequence for $\pi\colon\hat{\cy}\to\cx$.
This has $E_{2}$-page \begin{equation*}
\xymatrix@R-2pc@C-1pc{H^2(\hat{Y}_0) \\ H^1(\hat{Y}_0) \ar [rrd]^{d_2} \\ H^0(\cx) & 0 & H^2(\cx) & 0 & H^4(\cx)}
\end{equation*}with $d_{2}$ injective, and $H^{2}(\hat{\cy})\cong H^{2}(\hat{Y}_{0})\oplus H^{2}(\cx)/\image(d_{2})$.
(Note that $H^{2}(\cx)/\image(d_{2})$ and $H^{2}(\cy)$ both have
Hodge numbers $(0,10,0)$.) So degeneration at $E_2$ fails.

On the other hand, the Leray spectral
sequence for $f\colon \cx \to \PP^1$ takes the form 
\begin{equation*}
\xymatrix@R-2pc@C-1pc{H^0(R^2 f_*\QQ) & 0 & \QQ(-2)\\ 0 & H^1(R^1 f_*\QQ) & 0 \\ \QQ & 0 & H^2(R^0  f_*\QQ)}
\end{equation*}
with $d_{2}$ zero. However, the resulting Leray filtration on $H^2(\cx)$ is non-split in the category of MHS.  To see this, remark that:
\begin{itemize}[leftmargin=0.5cm]
\item $H^{2}(R^{0} f_{*}\QQ)\cong \QQ(-1)$ is generated
by the class of a smooth fiber;
\item $H^{1}(R^{1}f_{*}\QQ)\cong H^{1}(Y_{0})$ with $Y_0$ a smooth elliptic curve; and
\item $H^{0}(R^{0}f_{*}\QQ)\cong\QQ(-1)^{\oplus9}$ is generated by 8 components of fibers other than $X_{0}$, and one cocycle $\kappa\in\text{Cone}(C^{\bullet}(\hat{\cy})\to C^{\bullet}(\hat{Y}_{0}))$ given by $(P-\sigma,\wp)$, where $\wp$ is a path on $\hat{Y}_{0}$ from $0$ to $p$. 
\end{itemize}
Writing $\Omega^{1}(\hat{Y}_{0})=\CC\langle\omega_{0}\rangle$,
$\int_{\wp}\omega_{0}=\text{AJ}(p-0)$ gives the (nontorsion) extension
class of $[\kappa]$ by $H^{1}(Y_{0})$ in $H^{2}(\cx)$, and hence
of $\gr_{\cl}^{0}$ by $\gr_{\cl}^{1}$. Of course, Poincar\'e duality
also fails for $H^{2}(\cx)$.
\end{example}
As we shall see, one gets better behavior on all fronts by using perverse
Leray filtrations and intersection complexes.

\section{Perverse Leray}\label{sec-leray}

We begin by stating the Decomposition Theorem (DT) for a projective morphism $f:\cx\to\cs$ of complex algebraic varieties (of relative dimension $n=d_{\cx}-d_{\cs}$).  
Let $\ck\in D_c^b (\cx^{(\an)})$ [resp. $D^b \mathrm{MHM}(\cx)$] be a complex of sheaves of abelian groups [resp. mixed Hodge modules] which is constructible with respect to some stratification $\str$. 
Assume that $\ck$ is \textbf{semisimple} in the sense of being a direct sum of shifts of (semi)simple perverse sheaves $\IC_{\mathcal{Z}}(\mathbb{L})$ [resp. polarizable Hodge modules $\mathrm{MH}_{\mathcal{Z}}(\mathcal{L})$].\footnote{Here $\mathbb{L}$ [resp. $\mathcal{L}$] is a (semi)simple local system [resp. polarizable VHS] on a Zariski open in $\mathcal{Z}$. For uniformity of notation we shall use the notation $\IC_{\mathcal{Z}}(\mathbb{L})$ to refer to both perverse sheaves and Hodge modules.} References for the following statement are \cite[Thm. 6.2.5]{BBD}, \cite[Thm. 2.1.1]{dCM2}, and \cite{dC} for the perverse sheaf version, resp. \cite[Thms. 0.1-0.3]{toh} and \cite[(4.5.4)]{mhm} for the MHM version.
\begin{thm}[Decomposition Theorem] \label{th1.2}
\begin{enumerate}[label={\bf (\alph*)}, leftmargin=0.8cm]
\item[]
\item Writing $\pr^{j}f_{*}$ for $^{p}\mathcal{H}^{j}Rf_{*}$,
we have \begin{equation}\label{eq1.2a}
R f_* \ck \;\underset{\nc}{\simeq} \;\oplus_j (\pr^j f_* \ck)[-j] \;\simeq\; \oplus_{j,d} \IC_{Z_d}(\VV^j_d(\ck)[d])[-j]
\end{equation}as \emph{(}up to shift\emph{)} perverse sheaves \emph{{[}}resp. polarizable
Hodge modules\emph{{]}}, for some local systems \emph{{[}}resp. polarizable
VHS\emph{{]}} $\VV_{d}^{j}(\ck)$ on $Z_{d}\backslash Z_{d-1}$ \emph{(}smooth
of dimension $d$\emph{)}. Moreover, the $\pr^{j}f_{*}\ck$ are \textbf{semisimple}
perverse.

\item If $h$ is the class of a relatively ample line bundle
on $\cx$ and $\ck$ is perverse, then multiplication by $h^{j}$
induces an isomorphism \begin{equation}\label{eq1.2b}
\pr^{-j}f_* \ck (-j) \overset{\simeq}{\longrightarrow} \pr^j f_*\ck
\end{equation}for each $j\geq0$.
\end{enumerate}
\end{thm}
\begin{rem}
\label{rm1.2}{\bf (i)} If $\ck$ is a (pure) Hodge module of weight
$w$, the $\pr^{j}f_{*}\ck$ {[}resp. $Rf_{*}\ck${]} is pure of weight
$w+j$ {[}resp. $w${]}, and $\VV_{d}^{j}(\ck)$ underlies a VHS of
weight $w+j-d$.

{\bf (ii)} In the key special case where $\ck=\IC_{\cx}$ (which
is $\QQ_{\cx}[d_{\cx}]$ if $\cx$ is smooth), we write $\VV_{d}^{j}(\ck)=:\VV_{d}^{j}$.
In view of \cite{toh}, we still have \eqref{eq1.2a} in this
case when we relax the hypotheses on $f:\cx\to\cs$ to: $f$ proper,
$\cx$ Fujiki class C (dominated by K\"ahler).

{\bf (iii)} Although $\QQ_{\cx}[d_{\cx}]$ is perverse as long as
$\cx$ has local complete intersection singularities, it may not be
semisimple (and the DT may not apply). See Example \ref{ex1.6a} below.

{\bf (iv)} When restricting $f$ to an open analytic subset of the base such as a polydisk $\Delta^r$ (as we shall do below), \eqref{eq1.2a}-\eqref{eq1.2b} still hold though the $\VV_d^j (\ck)$ may not be semisimple \emph{on} $\Delta^r$.  For this reason (only), one should not call the $\pr^j f_*\ck$ ``semisimple'' in this setting.  Instead we shall say that they \emph{decompose}.

{\bf (v)} Over a disk, a weak form of the DT (first $\simeq$ of \eqref{eq1.2a}) holds without the semisimplicity constraint on $\ck$; see the Appendix by M. Saito.
\end{rem}

\begin{rem}
When $\cx\to \cs=\PP H^0(\mathrm{X},\co(L^{\otimes m}))$ ($L^{\otimes m}$ very ample) is the universal hypersurface section of a smooth $2D$-fold $\mathrm{X}$, the \emph{perverse weak Lefschetz theorem} \cite[Thm. 5.2]{BFNP} says that $\VV^j_d =0$ unless $j=0$ or $d=d_{\cs}$.  Moreover, $\VV^j_{d_{\cs}}$ is constant $\equiv H^{j+(2D-1)}(\mathrm{X})$ if $j\neq 0$, and $\VV_{d_{\cs}-1}^0 =0$ for $m\gg 0$. This plays a key r\^ole in producing singularities in normal functions associated to $D$-dimensional cycles on $\mathrm{X}$.
\end{rem}

Taking hypercohomology of \eqref{eq1.2a} yields a decomposition \begin{equation} \label{eq1.2c}
\HH^k(\cx,\ck)\;\cong\; \oplus_i \HH^i(\cs,\pr^{k-i}f_*\ck)\;\cong\;\oplus_{i,d}\; \IH^{d+i}(Z_d,\VV_d^{k-i}(\ck))
\end{equation}on the level of mixed Hodge structures. The \emph{perverse Leray filtration}
induced by \begin{equation} \label{1.2d}
\pl^{-\alpha} Rf_*\ck \; :=\; \oplus_{j\leq\alpha} (\pr^j f_*\ck)[-j]
\end{equation}is simply \begin{equation} \label{1.2e}
\pl^{\beta}\HH^k (\cx,\ck)\; :=\; \HH^k(\cs,\pl^{\beta-k}Rf_*\ck)\;\cong\; \oplus_{i \geq \beta} \HH^i(\cs,\pr^{k-i} f_*\ck).
\end{equation}
That is, under our hypotheses (of semisimplicity for $\ck$ and projectivity for $f$), the \emph{perverse} Leray spectral sequence $\HH^i(\cs,\pr^j f_* \ck)\implies\HH^*(\cx,\ck)$  converges at $E_2$, and the resulting (perverse Leray) filtration on $\HH^*(\cx,\ck)$ is split in the category of MHS.  This stands in marked contrast to the scenario in Example \ref{ex1.1}.
\begin{example}\label{ex2.1}
For any variety $\fx$, the morphism $\QQ_{\fx}[d_{\fx}]\to \IC_{\fx}$ in $D^b\mathrm{MHM}(\fx)$ induces a MHS map $H^k(\fx)\to\IH^k(\fx)$, and for a projective resolution of singularities $\tilde{\fx} \overset{\pi}{\twoheadrightarrow} \fx$ with compact exceptional divisor, $\tfrac{H^k(\fx)}{W_{k-1}}\overset{\pi^*}{\hookrightarrow}H^k(\tilde{\fx})$ (by the same proof as \cite[Thm. 5.41]{PS}).  Theorem \ref{th1.2} guarantees that $\IC_{\fx}$ is a direct summand of $R\pi_*\QQ_{\tilde{\fx}}[d_{\fx}]$, so that $\tfrac{H^k(\fx)}{W_{k-1}}\hookrightarrow \IH^k(\fx)\hookrightarrow H^k(\tilde{\fx})$.  For $\fx$ compact, the first term becomes $\mathrm{Gr}^W_k H^k(\fx)$, and the first injection is \cite[Thm. 3.2.1]{dCMS}.

In particular, for $\pi\colon \hat{\cy}\to \cx$ as in Example \ref{ex1.1}, we have $\gr^W_2 H^2(\cx) \cong \IH^2(\cx) \cong \QQ(-1)^{\oplus 10}$ and $H^2(\hat{\cy})\cong \QQ(-1)^{\oplus 11}$.  Writing $f^{\text{sm}}\colon f^{-1}(U)\to U$ for the smooth part of $f$ and $\ch^i := R f^{\text{sm}}_*\QQ$, Theorem \ref{th1.2} applies to:
\begin{itemize}[leftmargin=0.5cm]
\item \underline{$\ck = \QQ_{\hat{\cy}}[2]$ and $\pi$} $\implies$ $R\pi_*\ck \simeq \pr^0 \pi_* \ck \simeq \IC_{\cx}\oplus \imath^q_*\QQ(-1)$ (cf. \eqref{eq1.6a}-\eqref{eq1.6b}), so that $\pl_{\pi}^{\bullet}$ on $H^2(\hat{\cy})$ is trivial; and
\item \underline{$\ck=\IC_{\cx}$ and $f$} $\implies$ $Rf_* \ck \simeq \oplus_{j=-1}^1 (\pr^j f_* \ck)[-j]$ with $\pr^j f_* \ck \simeq \IC_{\PP^1}(\ch^{j+1}[1])$ ($j=\pm1$) and $\pr^0 f_* \ck \simeq \IC_{\PP^1}(\ch^1 [1])\oplus \bigoplus_{\sigma\in \PP^1\setminus U}\IH^2_{\pha,\sigma}$ (cf. \eqref{p10d}).  The graded pieces of $\pl^{\bullet}_f$ on $\IH^2(\cx)$ (cf. \eqref{eq1.4d} and \eqref{eq1.4h}) are then $\gr^{-1}_{\pl}\cong H^0(U,\ch^2)\cong \QQ(-1)$ (class of a section), $\gr^1_{\pl}\cong H^2_c (U,\ch^0)\cong \QQ(-1)$ (class of a smooth fiber), and $\gr_{\pl}^0 \cong \oplus_{\sigma}\IH^2_{\pha,\sigma}\cong \QQ(-1)^{\oplus 8}$ (from singular fibers; $\IH^1(\PP^1,\ch^1)$ vanishes).
\end{itemize}
Theorem \ref{th1.2} does \emph{not} apply to $\ck =\QQ_{\cx}[2]$ and $f$; see Example \ref{ex1.6a}.
\end{example}

We now look more systematically at immediate consequences of the DT for families over
a curve and resolutions of isolated singularities.

\section{Decomposition Theorem over a curve (1): Nearby and vanishing cycles}\label{S1.3}

Consider the scenario\begin{equation}\label{eq1.3*}
\xymatrix{& \cx_U \ar @{^(->} [r]^{\mathcal{J}} \ar[d]^{\fsm} & \cx \ar[d]^f & \cx_{\Sigma} \ar[d] \ar @{_(->}[l]_{\ci} \ar @{=}[r] & \cup_{\sigma\in\Sigma} \Xs \\ \cs\setminus\Sigma \ar@{=}[r] & U \ar @{^(->}[r]^{\jmath} & \cs & \Sigma \ar@{_(->}[l]_{\imath} }
\end{equation}where $d_{\cs}=1$ ($\implies d_{\cx}=n+1$), $\cs$ is smooth, $\cx_{U}$
and $\str|_{\cx_{U}}$ are topologically locally constant (e.g. equisingular)
over $U$, and our semisimple $\ck$ belongs to $\mathrm{MHM}(\cx)$ (i.e. its underlying complex is perverse).
For each $j$, we have \begin{equation} \label{eq1.3a}
\pr^j f_* \ck = \jmath_* \VV^j(\ck)[1]\oplus\bigoplus_{\sigma\in\Sigma} \imath^{\sigma}_* W^j_{\sigma}(\ck),
\end{equation}where $\VV^{j}(\ck)=R^{j-1}\fsm_{*}\ck|_{\cx_{U}}$ are local systems/VMHS
and $W_{\sigma}^{j}(\ck)$ are vector spaces/MHS.
Note that by $\Xs$ (and later, $X_{0}$) we always mean the \emph{reduced} special fiber, since MHM live on a complex analytic space.\footnote{Warning: the (``nonreduced'') components of $\Xs$ along which a local coordinate $t$ has order $>1$ are philosophically part of the singularity locus of $\Xs$, e.g. when considering support of $\pphi_t \QQ_{\cx}$.  See Prop. \ref{prop1.4} below.}

Writing $t$ for (the composition of $f$ with) a local coordinate
on a small disk $\Delta_{\sigma}\subset\cs$ about $\sigma$, the
associated nearby and vanishing cycle functors sit in a dual pair of distinguished
\emph{vanishing cycle triangles}\footnote{See \cite{{D3}} for the first and \cite[5.2.1]{mhp} for the second in the form used here.  Note that $\pphi_t :=\phi_t [-1]$ and $\ppsi_t := \psi_t [-1]$ send $\mathrm{MHM}(\cx)$ to $\mathrm{MHM}(X_{\sigma})$.} 
\begin{equation} \label{eq1.3b}
\xymatrix{\ci_{\sigma}^* \ar[rr]^{\sp} && \psi_t \ar[ld]^{\can} \\ & \phi_t \ar[ul]_{+1}^{\delta}}\;\;\;\;\text{and}\;\;\;\;\xymatrix@C-1pc{\ci^!_{\sigma} \ar[rr]^{\delta^{\vee}} && \pphi_t  \ar[ld]^{\var} \\ & \ppsi_t (-1) \ar[ul]_{+1}^{\sp^{\vee}} ,}
\end{equation}and satisfy $R\Gamma\ci_{\sigma}^{!}=\imath_{\sigma}^{!}Rf_{*}$,
$R\Gamma\ci_{\sigma}^{*}=\imath_{\sigma}^{*}Rf_{*}$, $R\Gamma\psi_{t}=\psi_{t}Rf_{*}$,
and $R\Gamma\phi_{t}=\phi_{t}Rf_{*}$ \cite{Ma}. Applied to $\ck$,
each morphism in the triangles yields a morphism of MHM, with the
exception of $\var$:\footnote{The tilde reflects the fact that, while related, $\var$ is not the standard $\mathtt{var}$ in the theory of perverse sheaves, because we do not have $\var\circ \can=T-I$ (see below).} here one needs to break $\psi_{t}\ck$ and $\phi_{t}\ck$
into unipotent and non-unipotent parts for the action of $T_{\sigma}$,
whereupon $\var^{u}:\phi_{t}^{u}\to\psi_{t}^{u}(-1)$ and $\var^{n}:\phi_{t}^{n}\to\psi_{t}^{n}$
induce MHM maps. Here $\var^u$ [resp. $\var^n$] is the morphism $\mathrm{Var}$ in \cite[3.4.10]{mhp} [resp. a canonical isomorphism], and $(-)^u = \ker (T^{\text{ss}}-I)$ [resp. $(-)^n =\text{im}(T^{\text{ss}}-I)$] is written $(-)_1$ [resp. $(-)_{\neq 1}$] in the notation of \textit{op. cit.}  (see also \cite[$\S\S$8-9]{Sn}).

Next, setting $\VV_{\lm}^{\ell}(\ck):=\psi_{t}\jmath_{*}\VV^{\ell}(\ck)$,
we have the monodromy invariants $\VV_{\lm}^{\ell}(\ck)^{T_{\sigma}}:=\ker(T_{\sigma}-I)$
and coinvariants $\VV_{\lm}^{\ell}(\ck)_{T_{\sigma}}:=\mathrm{coker}(T_{\sigma}-I)$.
By the DT, we compute \begin{equation} \label{eq1.3c}
\begin{split}
H^{\ell}(\Xs,\ck)&:= \HH^{\ell}(\Xs,\ci_{\sigma}^*[-1]\ck)\cong H^{\ell-1}(\imath_{\sigma}^*Rf_*\ck)\\
&=\oplus_jH^{\ell-j-1}(\imath_{\sigma}^* \pr^j f_*\ck)\\
&\cong \VV_{(\lm)}^{\ell}(\ck)^{T_{\sigma}} \oplus W_{\sigma}^{\ell-1}(\ck)
\end{split}
\end{equation}for the \emph{special fiber cohomology} and \begin{equation} \label{eq1.3d}
\begin{split}
H^{\ell}_{\Xs}(\cx,\ck)&:= \HH^{\ell}(\Xs,\ci^!_{\sigma}[-1]\ck)\cong H^{\ell-1}(\imath_{\sigma}^! Rf_*\ck)\\
&= \VV^{\ell-2}_{(\lm)}(\ck)_{T_{\sigma}}(-1)\oplus W_{\sigma}^{\ell-1}(\ck)
\end{split}
\end{equation}for the \emph{special fiber ``homology''}, where we used $\imath_{\sigma}^{*}\jmath_{*}\VV=H^{0}(\Delta_{\sigma},\jmath_{*}\VV)=\VV^{T_{\sigma}}$
and $\imath_{\sigma}^{!}\jmath_{*}\VV=H_{c}^{2}(\Delta_{\sigma},\jmath_{*}\VV)[-2]\cong H^{1}(\Delta_{\sigma}^{*},\VV)[-2]=\VV_{T_{\sigma}}(-1)[-2].$
We also write \begin{equation} \label{eq1.3e}
\begin{split}
H^{\ell}_{\lm,\sigma} (\ck) &:= \HH^{\ell}(\Xs,\ppsi_t\ck)\cong H^{\ell-1}(\psi_t Rf_*\ck)\\ &= \VV_{\lm}^{\ell}(\ck)
\end{split}
\end{equation}for the \emph{limiting cohomology} and \begin{equation} \label{eq1.3f}
H^{\ell}_{\van,\sigma}(\ck):= \HH^{\ell}(\Xs,\pphi_t\ck) \cong H^{\ell-1}(\phi_t Rf_*\ck)
\end{equation}for the \emph{vanishing cohomology}. These spaces carry natural MHSs
with morphisms induced by the MHM-maps above; we can either break
$\var:H_{\van,\sigma}^{\ell}(\ck)\to H_{\lm,\sigma}^{\ell}(\ck)$
into unipotent and non-unipotent parts, or regard it as a map of $\QQ$-vector
spaces --- one whose composition $\var\circ\can\in\mathrm{End}_{\QQ}(H_{\lm,\sigma}^{\ell}(\ck))$
with $\can$ yields $N_{(H^{\ell}_{\lm})^u}\oplus \mathrm{Id}_{(H^{\ell}_{\lm})^n}$. While not a morphism of MHS (since $N$ is a $(-1,-1)$ morphism on $(H^{\ell}_{\lm})^u$),
the kernel {[}resp. cokernel{]} of the latter (which is the same as the kernel [resp. cokernel] of $T_{\sigma}-I$) is a sub- {[}resp. quotient-{]}
MHS of $H_{\lm,\sigma}^{\ell}(\ck)$.

It remains to better understand $H_{\van,\sigma}^{\ell}(\ck)$ and
$W_{\sigma}^{j}(\ck)$. For \emph{any} (not necessarily semisimple)
perverse sheaf $\cp$ on $\cs$, sub- resp. quotient- objects of $\cp$
supported on $\{\sigma\}$ correspond to $\ker(\var)$ resp. $\mathrm{coker}(\can)$
on $\pphi_{t}\cp$ \cite{Sn}. So for $\cp$ semisimple, we have $\pphi_{t}\cp=\ker(\var)\oplus\image(\can)$,
which (together with $\ppsi_{t}\overset{\can}{\to}\pphi_{t}\to\imath_{\sigma}^{*}\overset{+1}{\underset{\sp}{\to}}$
and $\var\circ\can=N\oplus I$) yields identifications \begin{equation} \label{eq1.3g}
\left\{
\begin{split}
\image(\can)&\cong \mathrm{coim}(\var) \cong \ppsi_t\cp /\ker(T_{\sigma}-I) \\
\ker(\var) &\cong \mathrm{coker}(\can)\cong \ker\{ \sp:\imath_{\sigma}^* \to \ppsi_t \cp [1]\} .
\end{split}
\right.
\end{equation}If $\cp=\pr^{j}f_{*}\ck$, then $\image(\can)\cong\VV_{\lm}^{j}(\ck)/\ker(T_{\sigma}-I)$
while $\sp$ maps $\VV_{\lm}^{j}(\ck)^{T_{\sigma}}[1]\oplus W_{\sigma}^{j}(\ck)\to\VV_{\lm}^{j}(\ck)[1]$;
we conclude that \begin{equation} \label{eq1.3h}
\begin{split}
H^{\ell}_{\van,\sigma}(\ck) &= \oplus_j H^{\ell-j}(\pphi_t \pr^j f_* \ck) \\
&=\frac{\VV_{\lm}^{\ell}(\ck)}{\ker(T_{\sigma}-I)}\oplus W_{\sigma}^{\ell}(\ck).
\end{split}
\end{equation}

\begin{rem}
To put this more simply, a perverse sheaf $\mathsf{P}$ on $\Delta$ decomposes (\`a la Remark \ref{rm1.2}(iv)) $\iff$ $\pphi_t \mathsf{P}=\ker(\var)\oplus\mathrm{im}(\can)$ $\iff$ $\mathsf{P}$ takes the form $\imath_* W\oplus \jmath_* \VV[1]$ $\iff$ the corresponding quiver representation $\ppsi_t \mathsf{P} \underset{\var}{\overset{\can}{\rightleftarrows}} \pphi_t \mathsf{P}$ takes the form $V\underset{(N\oplus I,0)}{\overset{(\eta,0)}{\rightleftarrows}} \frac{V}{V^T} \oplus W.$ The decomposition \eqref{eq1.3h} is precisely this statement for $\mathsf{P}=(\pr^j f_* \ck)|_{\Delta}$, which decomposes by Theorem \ref{th1.2}.
\end{rem}

Finally, consider the composition \begin{equation} \label{eq1.3i}
H_{\Xs}^{\ell}(\cx,\ck) \overset{\delta^{\vee}}{\to} H^{\ell-1}_{\van,\sigma}(\ck) \overset{\delta}{\to} H^{\ell}(\Xs,\ck),
\end{equation}in which both maps are the identity on $W_{\sigma}^{\ell-1}(\ck)$
(and zero on the other summand). If $\cx_{\Delta_{\sigma}}=f^{-1}(\Delta_{\sigma})$
and $p\in\Delta_{\sigma}\backslash\{\sigma\}$, then \eqref{eq1.3i}
is really just the map $H_{\Xs}^{\ell}(\cx_{\Delta_{\sigma}})\to H^{\ell}(\cx_{\Delta_{\sigma}},X_{p})\to H^{\ell}(\Xs)$
with $\ck[-1]$-coefficients; the composite is the same if the middle
term is replaced by $H^{\ell}(\cx)$. Defining the \emph{phantom cohomology
at} $\sigma$ by \begin{equation} \label{eq1.3j}
H^{\ell}_{\pha,\sigma}(\ck) :=\ker\left\{\sp: H^{\ell}(\Xs,\ck)\to H^{\ell}_{\lm,\sigma}(\ck)\right\},
\end{equation}we therefore have \begin{equation*}
\begin{split}
H^{\ell}_{\pha,\sigma}(\ck) &= W_{\sigma}^{\ell-1}(\ck)=\image(\delta\circ \delta^{\vee})\\
&= \image\left(H^{\ell}_{\Xs}(\cx,\ck)\overset{\ci^{\sigma}_*}{\to}\HH^{\ell}(\cx,\ck[-1])\overset{\ci_{\sigma}^*}{\to} H^{\ell}(\Xs,\ck)\right).
\end{split}
\end{equation*}Denote $\gy:=\ci_{\sigma}^{*}\circ\ci_{*}^{\sigma}$ in the sequel.

\section{Decomposition Theorem over a curve (2): Consequences}\label{SDT2}

Continuing for the moment with $\ck\in\mathrm{MHM}(\cx)$ semisimple
(but otherwise arbitrary), there are a couple
of different ways to relate the special fiber cohomology and the limiting
cohomology. The immediate consequence of the first triangle of \eqref{eq1.3b}
is the \emph{vanishing cycle sequence} (of MHS) \begin{equation} \label{eq1.4a}
\cdots \to H^{\ell}(\Xs,\ck)\overset{\sp}{\to}H^{\ell}_{\lm,\sigma}(\ck)\overset{\can}{\to}H^{\ell}_{\van,\sigma}(\ck)\overset{\delta}{\to}H^{\ell+1}(\Xs,\ck)\to\cdots,
\end{equation}which is useful whenever one has methods to compute $\phi_{t}\ck$,
a subject taken up in Part II.

Also evident from the identifications in $\S$\ref{S1.3} is the \emph{Clemens--Schmid
sequence} \small \begin{equation} \label{eq1.4b}
0\to H^{\ell-2}_{\lm,\sigma}(\ck)_{T_\sigma} (-1)\overset{\sp^{\vee}}{\to}H_{\Xs}^{\ell}(\cx,\ck)\overset{\gy}{\to}H^{\ell}(\Xs,\ck)\overset{\sp}{\to}H^{\ell}_{\lm,\sigma}(\ck)^{T_{\sigma}}\to 0 ,
\end{equation} \normalsize which does away with the vanishing cohomology. The \emph{local invariant
cycle theorem} expressed by surjectivity of $\sp$ can be seen more
briefly by just taking stalks on both sides of \eqref{eq1.3a}. 
\begin{rem}\label{rem5n}
\textbf{(i)} A more elegant approach to \eqref{eq1.4b} can be formulated in terms of the octahedral axiom (cf. \cite[Rm. 5.2.2]{mhp} and also the Appendix below); though one must still invoke the DT to get \eqref{eq1.3a} (equiv. (5.2.2.3) in \textit{loc. cit.}).

\textbf{(ii)} In general, if $\mathrm{X}$ is a complex analytic space, $F:\mathrm{X}\to \Delta$ is proper, and $\mathsf{K}\in D^b\mathrm{MHM}(\mathrm{X})$ is self-(Verdier-)dual, then \textit{(a)} $\pr^j F_*\mathsf{K}$ decomposes ($\forall j$) and \textit{(b)} C-S \eqref{eq1.4b} holds ($\forall \ell$) are equivalent.  (This follows at once from duality of $\can$ and $\var$.) Also see Theorem A.4 in the Appendix below.
\end{rem}
There are two amplifications that make \eqref{eq1.4b} more useful: first,
one can extend it to a longer sequence of MHS by using the unipotent
parts:
\[
H^{\ell}(\Xs,\ck)\overset{\sp}{\to}H_{\lm,\sigma}^{\ell}(\ck)^{u}\overset{N_{\sigma}}{\to}H_{\lm,\sigma}^{\ell}(\ck)^{u}(-1)\overset{\sp^{\vee}}{\to}H_{\Xs}^{\ell+2}(\cx,\ck)(-1).
\]
(Notice that $\{\ker(N_{\sigma})\subset H_{\lm,\sigma}^{\ell}(\ck)^{u}\}=\{\ker(T_{\sigma}-I)\subset H_{\lm,\sigma}^{\ell}(\ck)\}=\{\ker(\can)\subset H_{\lm,\sigma}^{\ell}(\ck)\}$.)
Second, the hard Lefschetz part of DT implies isomorphisms $\VV^{-j}(\ck)\overset{h^{j}}{\to}\VV^{j}(\ck)(j)$
and $W^{-j}(\ck)\overset{h^{j}}{\to}W^{j}(\ck)(j)$, hence \begin{equation} \label{eq1.4c}
H^{-\ell}_{\lm,\sigma}(\ck)\overset{\cong}{\underset{h^{\ell}}{\to}} H^{\ell}_{\lm,\sigma}(\ck)(\ell) \;\;\text{and}\;\; H^{-\ell+1}_{\pha,\sigma}(\ck)\overset{\cong}{\underset{h^{\ell}}{\to}} H^{\ell+1}_{\pha,\sigma}(\ck)(\ell).
\end{equation}

In addition to these local results, we mention one consequence of
a global flavor: the \emph{generalized Shioda formula}, which for
$\cs$ a complete curve reads \begin{equation}\label{eq1.4d}
\begin{split}
\HH^k(\cx,\ck)&\cong \oplus_{i=-1}^1 \gr^i_{\pl}\HH^k(\cx,\ck)\cong\oplus_{i=-1}^1 \HH^i(\cs,\pr^{k-i}f_*\ck)\\
&\cong\oplus_{i=-1}^1 \left\{\HH^{i+1}(\cs,\jmath_* \VV^{k-i}(\ck))\oplus \HH^i(\cs,\imath_* W^{k-i}(\ck))\right\}\\
&= H^0(U,\VV^{k+1}(\ck))\oplus \left\{\IH^1(\cs,\VV^k(\ck))\oplus\bigoplus_{\sigma}W_{\sigma}(\ck)\right\}\\ &\mspace{100mu}\oplus H^2_c(U,\VV^{k-1}(\ck)),
\end{split}
\end{equation}where we remark that the last term $\cong H^{0}(U,\VV^{k-1}(\ck))^\vee(-1)$.
On the other hand, if $\cs$ is a quasi-projective curve, the last
term is simply omitted. In either case, $\HH^{k}(\cx,\ck)$ surjects
onto the first ($i=-1$) term, i.e. the \emph{global invariant cycle
theorem} holds.

Now we specialize to the case $\ck=\IC_{\cx}$, noting (in light of
Remark \ref{rm1.2}(ii)) that we can relax the hypotheses on $f,\cx,\cs$
somewhat if we ignore the hard Lefschetz statements. By considering
that $\IC_{\cx}|_{\cx^{\sm}}=\QQ_{\cx^{\sm}}[n+1]$ on the smooth
part of $\cx$ (to get the degrees right), one arrives at identifications
$\HH^{k}(\cx,\IC_{\cx})=\IH^{k+n+1}(\cx)$, $H^{k}(\Xs,\IC_{\cx})=H^{k-1}(\imath_{\sigma}^{*}Rf_{*}\IC_{\cx})=\IH^{k+n}(\cx_{\Delta})$,
and $H_{\Xs}^{k}(\cx,\IC_{\cx})=H^{k-1}(\imath_{\sigma}^{!}Rf_{*}\IC_{\cx})=\IH_{c}^{k+n}(\cx_{\Delta})$.
Accordingly we write
\begin{equation} \label{p10d}
\IH_{\pha,\sigma}^{k+n}:=\image\left\{ \IH_{c}^{k+n}(\cx_{\Delta})\to\IH^{k+n}(\cx_{\Delta})\right\} =H_{\pha,\sigma}^{k}(\IC_{\cx}),
\end{equation}
and note that there are morphisms (of MHS) from $H^{k+n}(\cx_{\Delta})\cong H^{k+n}(\Xs)$
{[}resp. $H_{c}^{k+n}(\cx_{\Delta})\cong H_{n-k+2}(\Xs)(-n-1)$, $H_{\pha,\sigma}^{k+n}:=\image\{H_{n-k+2}(\Xs)(-n-1)\overset{\gy}{\to}H^{n+k}(\Xs)\}${]}
to $\IH^{k+n}(\cx_{\Delta})$ {[}resp. $\IH_{c}^{k+n}(\cx_{\Delta})$,
$\IH_{\pha,\sigma}^{k+n}${]}, which in general are neither injective
nor surjective.%
\footnote{One also has maps from $H^{*}(\Xs)\to\IH^{*}(\Xs)$, but $\IH^{*}(\Xs)$
is different from $\IH^{*}(\cx_{\Delta})$ (e.g., when $\cx_{\Delta}$
is smooth and $\Xs$ is not).%
} As the restriction of $\IC_{\cx}$ to a fiber $X_{p}$ over $p\in U$
is $\IC_{X_{p}}[1]$, we also write $\VV^{k}(\IC_{\cx})=:\mathcal{IH}_{f}^{k+n}$,
and \begin{equation*}
\left\{
\begin{split}
\IH^{k+n}_{\lm,\sigma} &:= \psi_t \jmath_* \mathcal{IH}_f^{k+n} = \HH^k(\Xs,\ppsi_t \IC_{\cx})=H^k_{\lm,\sigma}(\IC_{\cx}) \\
\IH^{k+n}_{\van,\sigma}&:= \phi_t\jmath_* \mathcal{IH}_f^{k+n} = \HH^k(\Xs,\pphi_t \IC_{\cx})=H^k_{\van,\sigma}(\IC_{\cx}).
\end{split}
\right.
\end{equation*}

With this notation (and $\ell=k-n$), \eqref{eq1.4a}--\eqref{eq1.4d}
become \begin{equation} \label{eq1.4e}
\cdots \to \IH^k(\cx_{\Delta})\overset{\sp}{\to}\IH^k_{\lm,\sigma}(X_t)\overset{\can}{\to}\IH^k_{\van,\sigma}(X_t)\overset{\delta}{\to}\IH^{k+1}(\cx_{\Delta})\to\cdots ,
\end{equation}\begin{equation} \label{eq1.4f}
0\to\IH^{k-2}_{\lm,\sigma}(X_t)_{T_{\sigma}}(-1)\overset{\sp^{\vee}}{\to}\IH^k_c(\cx_{\Delta})\to\IH^k(\cx_{\Delta})\overset{\sp}{\to}\IH^k_{\lm,\sigma}(X_t)^{T_{\sigma}}\to 0,
\end{equation}\begin{equation} \label{eq1.4g}
\IH^{n-k}_{\lm,\sigma}\overset{\cong}{\underset{h^k}{\to}}\IH^{n+k}_{\lm,\sigma}(k),\;\;\;\;\;\;\;\;\;\;\IH^{n-k+1}_{\pha,\sigma}\overset{\cong}{\underset{h^k}{\to}}\IH_{\pha,\sigma}^{n+k+1}(k),
\end{equation}\begin{equation} \label{eq1.4h}
\IH^k(\cx)\cong H^0(U,\mathcal{IH}^k_f)\oplus\{\IH^1(\cs,\mathcal{IH}^{k-1}_f)\oplus\bigoplus_{\sigma}\IH^k_{\pha,\sigma}\}\oplus H^2_c(U,\mathcal{IH}^{k-2}_f).
\end{equation}Note that $\ker(\can)=\IH_{\lm,\sigma}^{k}(X_{t})^{T_{\sigma}}$,
and \eqref{eq1.4e} splits at all but the $\IH_{\lm}$ terms.
\begin{rem}
To compute $\IH^k(\cx_{\Delta})$ (say, for use in \eqref{eq1.4f}) one needs to know $\ci_{\sigma}^* \IC_{\cx}$, which was studied by Dimca and Saito in \cite{DSa}.  In general, $\ci_{\sigma}^*\IC_{\cx}[-1]$ is perverse, and there is a map $c_{\sigma}\colon \QQ_{X_{\sigma}}[n]\to \ci_{\sigma}^*\IC_{\cx}[-1]$ with kernel $\ph^{-1}\QQ_{\cx}[n+1]$ and cokernel $W_n \ph^0 \QQ_{\cx}[n+1]$.  If $\QQ_{\cx}[n+1]$ (hence $\QQ_{X_{\sigma}}[n]$) is perverse (e.g. if $\cx$ has local complete intersection singularities) then the kernel vanishes and $[c_{\sigma}]\colon H^k(X_{\sigma})\to \IH^k(\cx_{\Delta})$ is injective on the top $\gr^W_k$.  (See also Example \ref{ex2.1} and Remark \ref{rem1.7}.)  If $\ci_{\sigma}^* \IC_{\cx} [-1]=\QQ_{X_{\sigma}}[n]$, then $\IC_{\cx}=\QQ_{\cx}[n+1]$ and one just uses \eqref{eq1.4j} instead of \eqref{eq1.4f}.
\end{rem}
If $\cx$ is smooth, then $\IC_{\cx}=\QQ_{\cx}[n+1]$ and we simply
replace $\IH$ resp. $\mathcal{IH}_{f}$ everywhere by $H$ resp.
$\mathcal{H}_{f}$, except for the parabolic cohomology group $\IH^{1}(\cs,\ch_{f}^{k-1})$
in \eqref{eq1.4h}. The rank of the latter may be computed by the
\emph{Euler-Poincar\'e formula}, which reads \begin{equation} \label{eq1.4i}
\mathrm{rk}(\IH^1(\cs,\VV))=\sum_{\sigma\in\Sigma}\mathrm{rk}(\VV_v/\VV_v^{T_{\sigma}})-\chi(\cs)\cdot \mathrm{rk}(\VV_v)+h^1(\cs)\cdot \mathrm{rk}(\VV_c),
\end{equation}if $\VV=\VV_{c}\oplus\VV_{v}$ is the splitting into fixed (constant)
and variable parts for the local system underlying a PVHS. Next, for the two exact sequences of MHS we have:
\begin{thm}\label{thMain}
For $\cx$ smooth, the Clemens--Schmid sequence reads 
\begin{multline}\label{eq1.4j}
0\to H^{k-2}_{\lm,\sigma}(X_t)_{T_{\sigma}}(-1)\overset{\sp^{\vee}}{\to}H_{2n-k+2}(\Xs)(-n-1)\\ \overset{\gy}{\to}H^k(\Xs)\overset{\sp}{\to}H^k_{\lm,\sigma}(X_t)^{T_{\sigma}}\to 0.
\end{multline}
while the vanishing cycle sequence is 
\begin{equation} \label{eq1.4k}
0\to H_{\pha,\sigma}^k \to H^k(\Xs)\overset{\sp}{\to}H^k_{\lm,\sigma}(X_t)\overset{\can}{\to}H^k_{\van,\sigma}(X_t)\to H^{k+1}_{\pha,\sigma}\to 0 .
\end{equation}
\end{thm}
\begin{proof}
These follow directly from \eqref{eq1.4e}-\eqref{eq1.4f} since $\Xs$ is a deformation retract of $\cx_{\Delta}$.
\end{proof}

\begin{rem}
One can actually prove \eqref{eq1.4j} without invoking Theorem \ref{th1.2} using the equivalence of \textit{(a)} and \textit{(b)} in Remark \ref{rem5n} (with $\mathsf{K}=\QQ_{\cx}[n+1]$), together with C-S for semistable degenerations.  If $$\xymatrix{\cy_{\Delta} \ar[r]_{\tilde{\pi}} \ar[d]_g & \cx_{\Delta}\ar [d]^f \\ \Delta \ar [r]_{\pi} & \Delta}$$is a semistable reduction ($\pi(t)=t^{\kappa}$), then \eqref{eq1.4j} for $g$ $\implies$ $R g_* \QQ_{\cy}$ decomposes $\overset{\pi\text{ finite}}{\implies}$ $R f_* R \tilde{\pi}_* \QQ_{\cy} = R \pi_* R g_* \QQ_{\cy}$ decomposes.  Since $\QQ_{\cx}$ is a direct factor of $R\tilde{\pi}_* \QQ_{\cy}$, this also decomposes, and so \eqref{eq1.4j} holds for $f$.  (In fact, this amounts to a direct \emph{proof} of the DT in this case.) We thank M. Saito for this remark; see also Remarks A.4(iii) in his Appendix.
\end{rem}

Finally, we record two important facts about the terms in the sequences \eqref{eq1.4j}-\eqref{eq1.4k}.  For our purposes here $\text{sing}(\Xs)$ contains any ``nonreduced'' components of $\Xs$ (where $\text{ord}(t)>1$).
\begin{prop}
\label{prop1.4} For $\cx$ smooth and $d_{\text{sing}}:=\dim(\sing(\Xs))\colon$

\textbf{\emph{(i)}} $H_{\pha,\sigma}^{k}:=\ker(\sp)=\delta(H_{\van,\sigma}^{k-1}(X_{t}))=\image(\gy)$
is pure of weight $k$ \emph{(}and level $\leq k-2$\emph{)} and a
direct summand of $H^{k}(\Xs)$; and

\textbf{\emph{(ii)}} $H_{\van,\sigma}^{k}(X_{t})$ (hence $H_{\pha,\sigma}^{k+1}$)
is zero outside the range $n-d_{\text{sing}}\leq k\leq n+d_{\text{sing}}$.\end{prop}
\begin{proof}
\emph{(i)}  The MHS $\mathsf{H}_1 := H_{2n-k+2}(\Xs)(-n-1)$ has weights $\geq k$, while $\mathsf{H}_2 := H^{k}(\Xs)$ has weights $\leq k$.
Therefore $\gy\colon \mathsf{H}_1 \to \mathsf{H}_2$ is split and factors through $\gr^W_k \mathsf{H}_1$, which (with complexification dual to $\gr^W_{2n-k+2} H^{2n-k+2}(\Xs)_{\CC} = H^{n,n-k+2}(\Xs)\oplus \cdots \oplus H^{n-k+2,n}(\Xs)$) has level $\leq k-2$.

\emph{(ii)}  As $\pphi_{t}\QQ_{\cx}[n+1]$ is supported on $\sing(\Xs)$,
its perversity implies the existence of a stratification $\hat{\str}_{\bullet}$
($\dim\hat{\str}_{q}=q$) such that the cohomology sheaves $\mathcal{H}^{j}(\pphi_{t}\QQ_{\cx}[n+1]|_{\hat{\str}_{q}\backslash\hat{\str}_{q-1}})$
vanish unless $-d_{\text{sing}}\leq j\leq-q$. Hence in the hypercohomology
spectral sequence 
\[
E_{2}^{i,j}=\HH^{i}(\sing(\Xs),\mathcal{H}^{j}(\pphi_{t}\QQ_{\cx}[n+1]))\,\implies\, H_{\van,\sigma}^{i+j+n}(X_{t}),
\]
all nonzero terms lie in $\{i\geq0,\, i+2j\leq0,\, j\geq-d_{\text{sing}}\}$,
so $E_{\infty}^{i,j}=0$ outside $-d_{\text{sing}}\leq i+j\leq d_{\text{sing}}$.\end{proof}
\begin{rem}
\label{rem1.4a} More generally, part (ii) and \eqref{eq1.4k} hold
(for $H_{\text{van}}^{k}$ only) if $\cx$ has local complete intersection
singularities since then $\QQ_{\cx}[n+1]$ is still perverse. (Note
that $\dim(\text{sing}(\cx))\leq d_{\text{sing}}$.) This is because
the derivation of \eqref{eq1.4a} made no use of the Decomposition
Theorem.\end{rem}
\begin{cor}
\label{cor1.4}If $\cx$ is smooth, then $\gr_{F}^{0}H^{k}(\Xs)\cong\gr_{F}^{0}H_{\lm}^{k}(X_{t})^{T_{\sigma}}$
and $W_{k-1}H^{k}(\Xs)\cong W_{k-1}H_{\lm}^{k}(\Xs)^{T_{\sigma}}$. 
\end{cor}

\section{Decomposition Theorem over a curve (3): Examples}\label{sec-examples}

If $\cx$ is smooth and $\Xs$ has only isolated singularities ($d_{\text{sing}}=0$),
then by Proposition \ref{prop1.4}(ii) the vanishing cycle sequence becomes
\small \begin{equation} \label{eq1.5a}
0\to H^n(\Xs)\overset{\sp}{\to}H^n_{\lm,\sigma}(X_t)\overset{\can}{\to}H^n_{\van,\sigma}(X_t)\overset{\delta}{\to}H^{n+1}(\Xs)\overset{\sp}{\to}H^{n+1}_{\lm,\sigma}(X_t)\to 0
\end{equation}\normalsize with $H^{k}(\Xs)\overset{\cong}{\to}H_{\lm,\sigma}^{k}(X_{t})$ in
all other degrees ($k\neq n,n+1$). In particular, $T_{\sigma}=I$
on $H_{\lm,\sigma}^{k}(X_{t})$ for $k\neq n$, while Clemens--Schmid
reduces $ $to \begin{equation} \label{eq1.5b}
\left\{
\begin{split}
0\to &H^{n-1}_{\lm,\sigma}(X_t)\to  H_{n+1}(\Xs)({-n}{-1})\to H^{n+1}(\Xs)\to H^{n+1}_{\lm,\sigma}(X_t)\to 0 \\
&\text{and} \mspace{100mu}H^n(\Xs)\cong H^n_{\lm,\sigma}(X_t)^{T_{\sigma}}.
\end{split}
\right.
\end{equation}
These sequences are also valid when $n=1$ and the curve $\Xs$ has nonreduced components:  certainly $H_{-1}(\Xs)(-2)\overset{\gy}{\to}H^1(\Xs)$ is zero and (assuming $f$ connected) $T_{\sigma}=I$ on $H^2_{\lm}$ and $H^0_{\lm}$, which gives \eqref{eq1.5b} hence \eqref{eq1.5a}. 

We illustrate \eqref{eq1.5a}--\eqref{eq1.5b} for two simple examples,
then relate $H_{\van}$ to ``tails'' appearing in the semistable
reduction process.
\begin{example}
\label{ex1.5a} Let $\cx\overset{f}{\to}\PP^{1}$ be a smooth minimal
elliptic surface with section, and singular fibers of types $2\mathrm{I}_{1}$,
$\mathrm{I}_{6}^{*}$, $\mathrm{II}$, and $\mathrm{IV^{*}}$ (e.g.,
obtained from base-change and quadratic twist of the elliptic modular
surface for $\Gamma_{1}(3)$). These have $m_{\sigma}=1$, $11$,
$1$, resp. $7$ components, with $H_{\pha,\sigma}^{2}\cong\QQ(-1)^{\oplus(m_{\sigma}-1)}$;
and $\deg(\mathcal{H}_{f,e}^{1,0})=\frac{1}{12}(2\cdot1+12+2+8)=2$
$\implies$ $X$ $K3$.  (Here $\mathcal{H}^{1,0}_{f,e}$ is Deligne's canonical extension of $\mathcal{H}^{1,0}_f:=R^1 f^{\text{sm}}_* \Omega^1_{\mathcal{X}}$ to $\PP^1$; see \cite[\S III]{GGK09} for the contributions of the Kodaira singular fiber types to its degree.)  In \eqref{eq1.4h}, the end terms are generated
by the class of the (zero-)section and a fiber, while $\IH^{1}(\PP^{1},\ch_{f}^{1})$
has rank $4$ hence Hodge numbers $(1,2,1)$. (This rank comes either
from Euler-Poincar\'e or from subtracting the Picard rank $\rho=2+\sum(m_{\sigma}-1)=18$
from $22$.)

The Hodge-Deligne diagrams for the first three terms of \eqref{eq1.5a}
($n=1$) are well-known for each of these four degenerations.  We display them in Figure \ref{fig1},
writing numbers for $h^{p,q}\neq1$, eigenvalues $\neq1$ of $T_{\sigma}^{\mathrm{ss}}$ in braces, 
and $N:=\log(T_{\sigma}^{\mathrm{un}})$.
\begin{figure}
\includegraphics[scale=0.7]{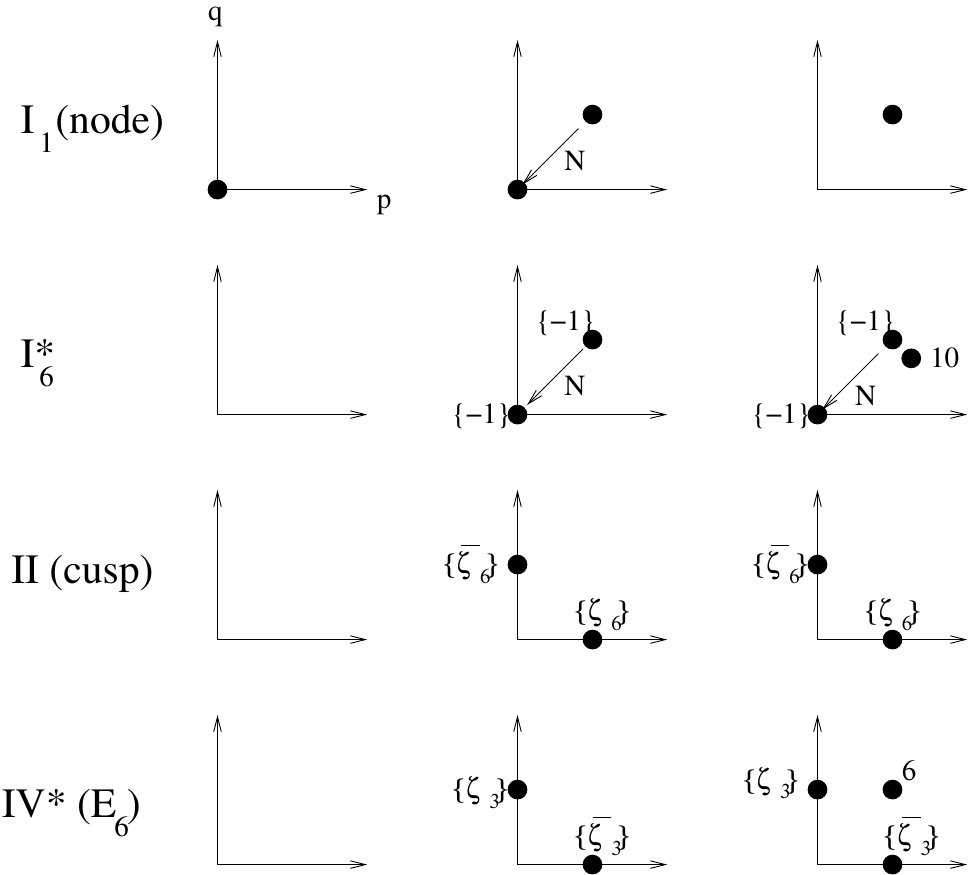}
\caption{Hodge-Deligne diagrams for Example \ref{ex1.5a}}\label{fig1}
\end{figure}
\end{example}

\begin{example}
\label{ex1.5b} Let $\cx_{\Delta}\overset{f}{\to}\Delta$ be a family
of $K3$ surfaces acquiring a single $\tilde{E}_{8}$ singularity:
locally, $f\sim x^{2}+y^{3}+z^{6}+\lambda xyz$. Then all $H_{\pha,0}^{k}$
are zero, and the first three terms of \eqref{eq1.5a} ($n=2$) are displayed in Figure \ref{fig2}.
\begin{figure}
\includegraphics[scale=0.7]{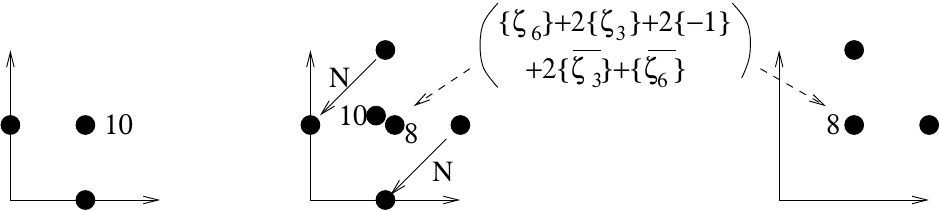}
\caption{Hodge-Deligne diagrams for Example \ref{ex1.5b}}\label{fig2}
\end{figure}
If $\fx_{\Delta}$ is the (singular) base-change by $t\mapsto t^{6}$,
then these terms are unchanged except that the action of $T^{\mathrm{ss}}$
trivializes -- which means that \eqref{eq1.4j} now \emph{fails}.
(As we shall see explicitly in Example \ref{ex1.6a}, $\QQ_{\fx_{\Delta}}[3]$
is not semisimple.) Of course \eqref{eq1.4e}--\eqref{eq1.4f} still
apply: in particular, $\IH^{2}(\fx_{\Delta})\cong H^{2}(X_{0})\oplus\QQ(-1)^{\oplus8}$
and $\IH_{\pha,0}^{2}=\{0\}$. 

On the other hand, performing a weighted blow-up of $f+t^{6}=0$ at
the origin yields the semistable/slc model $\cy_{\Delta}\overset{t}{\to}\Delta$,
with $Y_{0}=\tilde{X}_{0}\cup\cE$ and $\tilde{X}_{0}\cap\cE=:E$.
Here $E=\{f=0\}\subset\PP[1:2:3]$ is an elliptic curve and the ``tail''
$\cE=\{f+t^{6}=0\}\subset\PP[1:1:2:3]$ a del Pezzo surface of degree
$1$ ($\rho=9$). As we have seen in Example \ref{ex1.1}, the extension
class of $H^{2}(\tilde{X}_{0})$ by $H^{1}(E)$ in 
\[
0\to H^{1}(E)\to H^{2}(Y_{0})\to H^{2}(\tilde{X}_{0})\oplus H^{2}(\cE)\to H^{2}(E)\to0
\]
(i.e. in $H^{2}(X_{0})$) can be nontorsion. However, that of $\ker(H^{2}(\cE)\to H^{2}(E))$
by $H^{1}(E)$ is torsion due to the eigenspace decomposition under
the \emph{original} $T^{\mathrm{ss}}$ and the fact that we have not
altered $H_{\lm}^{1}$ in \eqref{eq1.5b}. This will change if we
take a more general pullback of the form
\[
t^{6}+a_{5}t^{5}z+t^{4}(a_{4}z^{2}+b_{4}y)+\cdots+x^{2}+y^{3}+z^{6}+\lambda xyz=0,
\]
as then $t\mapsto\zeta_{6}t$ no longer induces an automorphism of
$H^{2}(\cE,E)$.
\end{example}
We briefly explain how the relation between the ``tail'' $(\cE,E)$
and the vanishing cohomology generalizes for isolated singularities.
Consider the scenario \begin{equation}\label{eq1.5c}
\xymatrix{&\Delta\ar@{=}[r]&\Delta\ar[r]^{(\cdot)^{\kappa}}&\Delta\\&\cy \ar[u]^F \ar[r]^{\pi}&\fx \ar[u]\ar[r]^{\rho}&\cx\ar[u]^f\\ \tilde{X}_0 \cup\cE\ar@{=}[r]&Y_0 \ar@{^(->}[u]^{\tilde{\ci}}\ar[r]^{\pi_0}&X_0\ar@{^(->}[u]^{\ci}\ar@{=}[r]&X_0\ar@{^(->}[u]\\ \tilde{X}_0\cap \cE\ar@{=}[r]&E\ar@{^(->}[u]\ar[r]&\{p\}\ar@{^(->}[u]^{\imath_p}}
\end{equation}where $\cx,\cy,\tilde{X}_{0}$ are smooth, $X_{0}=(f)$ is reduced
and irreducible, $p\in\sing(X_{0})=\sing(\fx)$ is isolated, $\rho$
is cyclic base-change and $F$ is semistable ($\implies Y_{0}$ SNCD).
First we look at the case of $\cE,E$ irreducible and smooth:
\begin{prop}
\label{prop1.5a} As a mixed Hodge structure, $H_{\van}^{k}(X_{t})$
is the reduced cohomology $\tilde{H}^{k}(\cE\backslash E)$, and this
vanishes for $k\neq n$.\end{prop}
\begin{proof}
Applying $\pphi_{t}$ to the distinguished triangle in $D_{c}^{b}(\fx)$
\[
\cone(\alpha)[-1]\to\QQ_{\fx}[n+1]\overset{\alpha}{\to}R\pi_{*}\QQ_{\cy}[n+1]\overset{+1}{\to}
\]
yields a triangle in $D_{c}^{b}(X_{0})$ with terms: \begin{flalign*}
\bullet \; \pphi_{t}\cone(\alpha)[-1] &=\ci^{*}\cone(\alpha)[-1]=\cone(\ci^{*}\QQ_{\fx}[n]\to R\pi_{0*}\tilde{\ci}^{*}\QQ_{\cy}[n]) &&\\ &= \cone(\QQ_{X_{0}}[n]\to R\pi_{0*}\QQ_{Y_{0}}[n])&&\\
&=(\imath_{p})_{*}\cone(\QQ[n]\to R\Gamma\QQ_{\cE}[n])&&
\end{flalign*}using the fact that $\ppsi_{t}(\alpha)$ is an isomorphism;\begin{flalign*}
\bullet\;\pphi_{t}\QQ_{\fx}[n+1]&=\pphi_{t}\QQ_{\cx}[n+1]&&
\end{flalign*}since the base-change doesn't affect the first vanishing cycle triangle; and \begin{flalign*}
\bullet\;\pphi_{t}R\pi_{*}\QQ_{\cy}[n+1]&= R\pi_{0*}\pphi_{t}\QQ_{\cy}[n+1]=R\pi_{0*}\QQ_{E}(-1)[n-1]&&\\
&=\imath_{p*}R\Gamma\QQ_{E}(-1)[n-1],&&
\end{flalign*}where we used the fact that $\pphi_{t}$ of the constant sheaf is
$\QQ(-1)[-2]$ at a node (Example \ref{ex1.5a}). It is immediate
that $\pphi_{t}\QQ_{\cx}[n+1]$ is $(\imath_{p})_{*}$ of \begin{flalign*}
\cone (R\Gamma \QQ_E (-1)[-2]\oplus\QQ \to R\Gamma\QQ_{\cE})[n]&\simeq \cone(\QQ\to R\Gamma \QQ_{\cE\setminus E})[n] \\ &\simeq \oplus_k \tilde{H}^k (\cE\setminus E)[n-k];
\end{flalign*}which being perverse must vanish outside degree $0$.
\end{proof}
In the more general case where $\cE$ is a \emph{union} of smooth
$\{\cE_{i}\}_{i=1}^{e}$, $H_{\van}^{k}(X_{t})$ is still $0$ for
$k\neq n$ by Proposition \ref{prop1.4}(ii), but $H_{\van}^{n}(X_{t})$
is not as straightforward as in Proposition \ref{prop1.5a}. To see
what one \emph{can} say, write $\cE_{0}:=\tilde{X}_{0}$, $\cE_{I}:=\cap_{i\in I}\cE_{i}$,
$\cE_{i}^{*}:=\cE_{i}\backslash(\cup_{j\in\{0,\ldots,e\}\backslash\{i\}}\cE_{j})$,
\[
\cE^{[k]}:=\amalg_{\tiny\begin{array}{c}
I\subset\{1,\ldots,e\}\\
|I|=k+1
\end{array}}\cE_{I}\subset\amalg_{\tiny\begin{array}{c}
I\subset\{0,\ldots,e\}\\
|I|=k+1
\end{array}}\cE_{I}=:Y_{0}^{[k]}\overset{\ci^{[k]}}{\longrightarrow}Y_{0},
\]
and $\mu:=\dim(H_{\van}^{n}(X_{t}))$ for the \emph{Milnor number}.
\begin{thm}
\label{prop1.5b} \textbf{\emph{(i)}} The associated graded $\gr_{\bullet}^{W}H_{\van}^{n}(X_{t})$
is a subquotient of $$\bigoplus_{k=0}^{n}\left\{ H^{n-k}(\cE^{[k]})\oplus\left(H^{n-k}(Y_{0}^{[k]})\otimes\tilde{H}^{*}(\PP^{k})\right)\right\}.$$

\textbf{\emph{(ii)}} $\mu=(-1)^{n}\left\{ -1+\sum_{i=1}^{e}\chi(\cE_{i}^{*})\right\} .$\end{thm}
\begin{proof}
We shall see in Part II that for a semistable fiber at $t=0$,
$\pphi_{t}\QQ_{\cy}[n+1]\in\mathrm{Perv}(Y_{0})$ has terms in degrees
($k-n=$) $1-n$ thru $0$ with
\[
(\pphi_{t}\QQ_{\cy})^{k+1}=\ci_{*}^{[k]}(\oplus_{\ell=1}^{k}\QQ_{Y_{0}^{[k]}}(-\ell))=\ci_{*}^{[k]}\QQ_{Y_{0}^{[k]}}\otimes\tilde{H}^{*}(\PP^{k}).
\]
The rest of the proof of Prop. \ref{prop1.5a} is unchanged, and so
$\pphi_{t}\QQ_{\cx}[n+1]$ is $(\imath_{p})_{*}$ of the MHS ($H_{\van}^{n}(X_{t})=$)\begin{equation*}
\begin{split}
\HH^n\{&\cone(\QQ\oplus R\Gamma\,\pphi_t\QQ_{\cy}\to R\Gamma\QQ_{\cE})\}
\\
=&\;\HH^n\{\widetilde{R\Gamma}\QQ_{\cE}\to R\Gamma\QQ_{Y_0^{[1]}}\otimes\tilde{H}^*(\PP^1)\to\cdots\to R\Gamma\QQ_{Y_0^{[n]}}\otimes\tilde{H}^*(\PP^n)\}
\\
=&\;\HH^n\{\widetilde{R\Gamma}\QQ_{\cE^{[0]}}\to R\Gamma\QQ_{\cE^{[1]}}\oplus (R\Gamma\QQ_{Y_0^{[1]}}\otimes \tilde{H}^*(\PP^1))\to \\
&\mspace{200mu}\cdots\to R\Gamma\QQ_{\cE^{[n]}}\oplus (R\Gamma\QQ_{Y_0^{[n]}}\otimes\tilde{H}^*(\PP^n))\} ,
\end{split}
\end{equation*}giving (i). This also shows that $\chi(H_{\van}^{*}(X_{t}))+1$ ($=(-1)^{n}\mu+1$)
is given by\begin{equation*}
\begin{split}
\sum_{k=0}^n (-1)^k \{(\chi(\PP^k)-1)&\chi(Y_0^{[k]})+\chi(\cE^{[k]})\} \\
&=\sum_{i=1}^e\chi(\cE_i)+\sum_{k=1}^n (-1)^k(k\chi(Y_0^{[k]})+\chi(\cE^{[k]}))\\
&=\sum_{i=1}^e\chi(\cE_i^*),
\end{split}
\end{equation*}
which yields (ii).
\end{proof}
While $H_{\van}^{n}(X_{t})$ in our setting \eqref{eq1.5c} can in
general have weights from $0$ to $2n$, Theorem \ref{prop1.5b}(i)
makes it clear that the graded pieces are directly related to strata
of the tail, while (ii) is a close cousin of the theorem of A'Campo
\cite{AC}.  The proof of (i) actually yields a more precise computation of $\gr^W_{\bullet}H_{\van}^*$ related to the ``motivic Milnor fiber'' of \cite{DL}, and which we shall use systematically in Part II.
\begin{example}[{see also \cite[Sect. 6]{log2}}]
\label{ex1.5c} Suppose $X_{0}$ has a Dolgachev singularity of type
$E_{12}$, viz. $f\sim x^{2}+y^{3}+z^{7}$ locally. Taking $\kappa=42$
yields (for $\fx$) $t^{42}+x^{2}+y^{3}+z^{7}=0$, whose weighted
blow-up produces a singular fiber $X_{0}'\cup\cE'$ with $X_{0}'\cap\cE'\cong\PP^{1}\ni p_{1},p_{2},p_{6}$
and $X_{0}',\cE'$ having $A_{k}$ singularities at $p_{k}$. After
a toric resolution, we arrive at the SSR $\cy$, with $\cE_{1}=\widetilde{\cE'}$
a $K3$ surface and $\cE_{2},\ldots,\cE_{10}$ toric Fanos; $\cE_{1}$
meets $\tilde{X}_{0}$ and each $\cE_{i=2,\ldots,10}$ in a $\PP^{1}$,
and $H^{2}(\cE_{1}^{*})$ has Hodge numbers $(1,10,1)$. \[\includegraphics[scale=0.6]{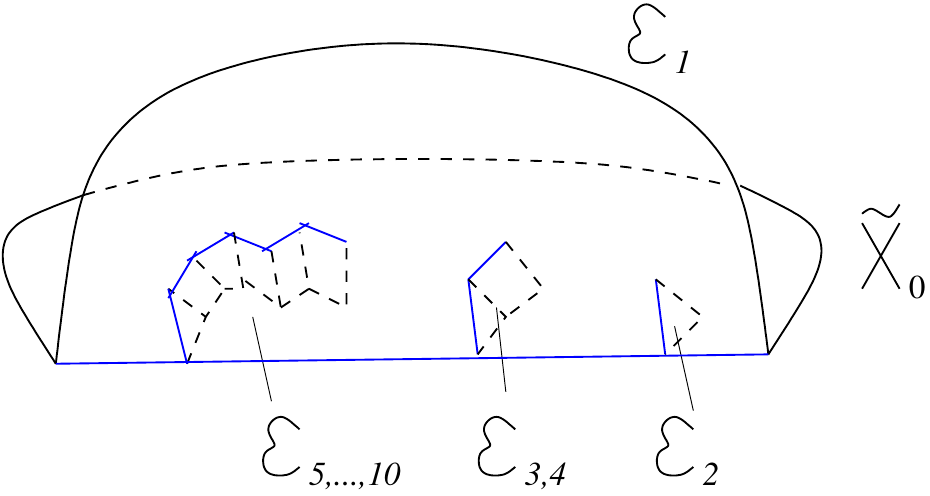}\]The
other $\cE_{i}^{*}$ are all $\mathbb{G}_{m}^{2}$ or $\mathbb{G}_{m}\times\mathbb{A}^{1}$,
so have $\chi=0$, which yields $\mu=\chi(\cE_{i}^{*})-1=12$; indeed,
$H_{\van}^{2}(X_{t})$ is just $H^{2}(\cE_{1}^{*})=H_{\mathrm{tr}}^{2}(\cE_{1})$
in this case. The moral is that toric components of $\cE$ arising
from resolving canonical singularities (here the threefold $A_{1}+A_{2}+A_{6}$
singularities) won't complicate the result much beyond the case of
$\cE$ smooth.  A general reason for this is given by Prop. \ref{propA} below.
\end{example}

We conclude by sketching a geometric application of the more general form \eqref{eq1.4b} of Clemens--Schmid, where $\ck$ is not $\IC_{\cx}$.  (Full details will appear elsewhere.)
\begin{example}\label{exKatz}
Let $\ck:=\IC_{\cx}(\tilde{\LL})$	, where:
\begin{itemize}[leftmargin=0.5cm]
\item $\cx\overset{f}{\to}\PP^1_t$ is the minimal smooth compactification of the elliptic curve family $w^2 = tz(z-1)(z+1)+t^2$ with $\Sigma = \{0,\pm\tfrac{2}{3\sqrt{3}},\infty\}$;
\item $\LL$ is the rank two local system on $\PP^1_z\setminus\{0,\pm 1,\infty\}$ arising from relative $H^1$ of the rational elliptic surface $\cE\to \PP^1_z$ with fibers $\mathrm{I}_2$ at $\pm1$ and $\mathrm{I}_4$ at $0,\infty$; and
\item 	$\tilde{\LL}:=\pi^* \LL$ is the pullback local system on $\cy:=\cx\setminus \overline{\pi^{-1}(0,\pm1,\infty)}\overset{\jmath}{\hookrightarrow}\cx$, where $\cx \overset{\pi}{\dashrightarrow}\PP^1_z$ is given by $\pi(t,w,z):=z$.
\end{itemize}
Taking $\ell=0$ in \eqref{eq1.4b}, one checks that for $\sigma \in \Sigma$, $\gy=0$ and so
\begin{equation}\label{K!}
H^0(X_{\sigma},\IC_{\cx}(\tilde{\LL}))\underset{\cong}{\overset{\sp}{\longrightarrow}} H^0_{\lm,\sigma}(\IC_{\cx}(\tilde{\LL}))^{T_{\sigma}}\cong H^2_{\text{tr}}(S_t)^{T_{\sigma}}_{\lm,\sigma}
\end{equation}
where $S_t :=\widetilde{X_t\times_{\PP^1_z}\cE}$ is a family of (smooth) surfaces over $\PP^1_t\setminus \Sigma$ introduced by Katz \cite{Katz}.

For $t\notin\Sigma$, $H^2(S_t)=\IH^1(X_t,\tilde{\LL})$ has rank $7$ by Euler-Poincar\'e \eqref{eq1.4i}, with Hodge numbers $(2,3,2)$.  Viewed as a weight-$2$ VHS on $\PP^1\setminus \Sigma$, it has geometric monodromy group $G_2$, as shown by an arithmetic argument in \textit{op. cit.} and by a direct calculation of the monodromies in \cite{Jrthesis}, both quite painstaking.  However, we can use \eqref{K!} to quickly deduce the Hodge-Deligne diagrams for $H^2_{\text{tr}}(S_t)_{\lm,\sigma}$ (Figure \ref{fig6.3}); in particular, at $\sigma=0$ and $\infty$ this is much easier than using a smooth compactification of the family of surfaces. 
\begin{figure}
\includegraphics[scale=0.7]{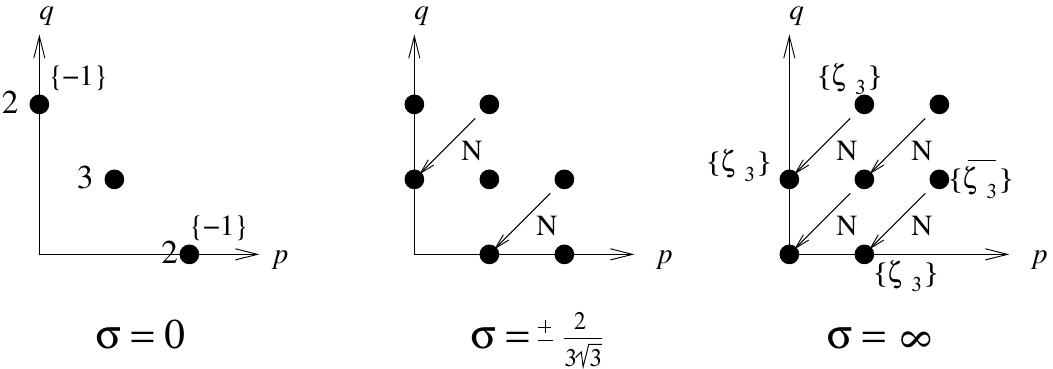}
\caption{Hodge-Deligne diagrams for Example \ref{exKatz}}\label{fig6.3}
\end{figure}
\vspace{3mm}

\hangindent=1em
\hangafter=1
\noindent\underline{$\sigma=\pm\tfrac{2}{3\sqrt{3}}$}:  $X_{\sigma}$ has Kodaira type $\mathrm{I}_1$, with node $p_{\sigma}\in Y_{\sigma}\overset{\jmath^{\sigma}}{\hookrightarrow}X_{\sigma}$, and $\ci^*_{\sigma}\ck=\jmath^{\sigma}_* (\tilde{\LL}|_{Y_{\sigma}})[1]$.  The LHS of \eqref{K!} is thus $H^1(\Xs,\jmath^{\sigma}_* \tilde{\LL}|_{Y_{\sigma}})$, which is an extension of $H^1(\tilde{X}_{\sigma},\rho^* \jmath^{\sigma}_* \tilde{\LL}|_{Y_{\sigma}})$ (weight 2 and rank 3 by \eqref{eq1.4i}) by $H^1(E_{p_{\sigma}})$, where $\rho\colon \tilde{X}_{\sigma}\to \Xs$ is the normalization and $E_{p_{\sigma}}$ is the fiber of $\cE$ over $\pi(p_{\sigma})$.
\vspace{3mm}

\noindent\underline{$\sigma=0$}: $X_0 = \cup_{i=0}^4 P_i$ has type $\mathrm{I}_0^*$, 5 $\PP^1$'s with $P_0$ meeting each $P_i$ in a single node $p_i$. The divisor $Z=\cx\setminus\cy$ contains $\{P_i\}_{i=1}^4$, and has additional components meeting $\{P_i\}_{i=2}^4$ at $\{q_i\}_{i=2}^4$. So $\ch^{-1}\ci^*_0\ck$ is $\jmath^0_* \tilde{\LL}|_{P_0\setminus \{p_i\}_{i=1}^4}$ on $P_0$ and $\QQ(0)$ on the $\{P_i\}_{i=1}^4$, while $\ch^0\ci^*_0\ck$ is $\imath^{q_i}_* \QQ(-1)$ on the $\{P_i\}_{i=2}^4$.  Conclude that $H^1(X_0,\ch^{-1}\ci_0^*\ck)=0$ by Mayer-Vietoris and \eqref{eq1.4i}, so that $\text{LHS}\eqref{K!}\cong H^0(X_0,\ch^0\ci^*_0\ck)\cong \QQ(-1)^{\oplus 3}$.  (The key observation here was that the pullback of a unipotent degeneration of weight 1 rank 2 HS by $(x,y)\mapsto xy$ has local $\IH^1\cong \QQ(-1)$, cf. \eqref{eq1.9.3}-\eqref{eq1.9.4}.) 
\vspace{3mm}

\noindent\underline{$\sigma=\infty$}:  $X_{\infty}=\cup_{i=1}^3 Q_i$ has type $\mathrm{IV}$, 5 concurrent $\PP^1$'s meeting at $p_{\infty}$.  One finds $Q_i\cap Z=\{p_0^i,p_1^i,p_{-1}^i,p_{\infty}\}$ for $i=1,2$, while $Z$ contains $Q_3$ and another component meeting it in a node $q$.  Therefore $\ch^{-1}\ci^*_{\infty}\ck$ is $\jmath^{\infty}_* \tilde{\LL}|_{Q_i\cap\cy}$ on $Q_1,Q_2$, and $\QQ(0)$ on $Q_3$; whereas $\ch^0\ch^*_{\infty}\ck$ is just $\imath^q_* \QQ(-1)$ on $Q_3$.  One gets a sequence \begin{multline}0\to H^1(X_{\infty},\ch^{-1}\ci_{\infty}^*\ck)\to \text{LHS}\eqref{K!} \\ \to H^0(X_{\infty},\ch^0 \ci^*_{\infty}\ck)\overset{d_2}{\to}H^2(X_{\infty},\ch^{-1}\ci_{\infty}^*\ck)\to 0\end{multline}with $d_2$ an isomorphism and first term $\QQ(0)$.
\vspace{3mm}

\hangindent=0em
\noindent To get $\ker(N)$ (instead of $\ker(T_{\sigma}-I)$) at $0$ and $\infty$, one performs a base-change by $t\mapsto t^2$ [resp. $t^3$] followed by a proper modification to replace $X_0$ [resp. $X_{\infty}$] by a smooth elliptic curve. The computations then proceed as above.
\end{example}

\section{Decomposition Theorem for an isolated singularity}\label{sec-isolated}

Let $\tilde{\fx}\overset{\pi}{\to}\fx$ be the resolution of an isolated
singularity $p\overset{\imath}{\hookrightarrow}\fx$ with exceptional
divisor $\cE$ (not assumed smooth or normal-crossings). With $d:=\dim(\fx)$,
\eqref{eq1.2a} specializes to\begin{equation} \label{eq1.6a}
R\pi_* \QQ_{\tilde{\fx}}[d] \simeq \IC_{\fx}\oplus \imath_* \left(\oplus_i V^i[-i]\right),
\end{equation}where the stalk cohomologies of the intersection complex $H^{j}(\imath^{*}\IC_{\fx})$
{[}resp. $H^{-j}(\imath^{!}\IC_{\fx})${]} vanish for $j>0$. Writing
$\cn_{\cE}$ for the preimage of a ball about $p$, we apply $H^{j}\circ\imath^{*}$
{[}resp. $H^{-j}\circ\imath^{!}${]} to \eqref{eq1.6a} to find\begin{equation}\label{eq1.6b}
V^i \cong \left\{ \begin{array}{cc}H^{d+i}(\cn_{\cE}), & i>0\\H_{c}^{d+i}(\cn_{\cE}), & i\leq0\end{array} \right. \cong \left\{ \begin{array}{cc}H^{d+i}(\cE), & i>0\\H_{d-i}(\cE)(-d), & i\leq0.\end{array} \right.
\end{equation}The hard Lefschetz property $V^{-j}({-j})\cong V^{j}$ (as $\QQ$-MHS)
therefore reads $H_{d+j}(\cE)({-d}{-j})\cong H^{d+j}(\cE)$, which
implies that the $V^{i}$ are all pure.
\begin{example}
\label{ex1.6a} Specialize the scenario \eqref{eq1.5c} to $f$ a
family of elliptic curves with cuspidal (type $\mathrm{II}$) fiber
$X_{0}$, and $\kappa=6$. The local equation of $\fx$ at $p$ is
then $x^{2}+y^{3}+t^{6}+\cdots$, i.e. an $\tilde{E}_{8}$ (simple
elliptic) singularity, with $\tilde{X}_{0}\cong\PP^{1}$ and $\cE$
a CM elliptic curve. We have $R\pi_{*}\QQ_{\cy}[2]\simeq\IC_{\fx}\oplus\imath_{*}\QQ(-1)$
and a short-exact sequence\begin{equation}\label{eq1.6c}
0\to \imath_* H^1(\cE)\to \QQ_{\fx}[2] \to \IC_{\fx}\to 0
\end{equation}in $\mathrm{Perv}(\fx)$, which we claim is \emph{not split}.
(Hence, as remarked in Example \ref{ex2.1}, the DT for $f$ does not apply to $\ck = \QQ_{\fx}[2]$.)

Indeed, were \eqref{eq1.6c} split, $\QQ_{\fx}[2]$ would be semisimple
hence (as in $\S$\ref{S1.3}) 
\[
\pphi_{t}\QQ_{\fx}[2]=\image(\can)\oplus\ker(\var)=\image(\ppsi_{t}\QQ_{\fx}[2])\oplus\ker(T^{6}-I),
\]
where everything is supported on $p$. But $T^{6}$ acts trivially
(since the eigenvalues of $T$ are $\zeta_{6}^{\pm1}$, cf. Example
\ref{ex1.5a}), while $\can$ is onto, a contradiction. (Alternatively,
one could take the long-exact hypercohomology sequence of \eqref{eq1.6c}
and observe that the connecting homomorphism $\delta:\IH^{1}(\fx)=\HH^{-1}(\IC_{\fx})\to H^{1}(\cE)$
is an isomorphism in view of \eqref{eq1.4f}.)

More generally, the argument shows that sequences like \eqref{eq1.6c}
are non-split if the order of a nontrivial eigenvalue of $T_{\mathrm{ss}}$
divides the base-change exponent $\kappa$. In particular, this applies
to the sequence $0\to\imath_{*}H^{2}(\mathcal{E}\backslash E)\to\QQ_{\fx}[3]\to\IC_{\fx}\to0$
implicit in Example \ref{ex1.5b}.  Since the DT then applies only to $\IC_{\fx}$ (not $\QQ_{\fx}[n+1]$), we have only \eqref{eq1.4f} (and not \eqref{eq1.4j}) for $\fx$.
\end{example}
As an immediate consequence of \eqref{eq1.6a}--\eqref{eq1.6b}, we
find that \begin{equation}\label{eq1.6d}
H^k(\tilde{\fx})\twoheadrightarrow H^k(\cE)\;\;\;\text{for}\;\;\; k\geq d.
\end{equation}Now suppose that $\fx=X_{0}$ appears as the singular fiber in a family
$\cx\to\Delta$ with $n=d$ and $\cx$ smooth (and write $\cE\subset\tilde{X}_{0}\overset{\pi}{\twoheadrightarrow}X_{0}$
for the exceptional divisor). For $k<n$, the Clemens--Schmid and vanishing cycle sequences give $H^{k}(X_{0})\cong H_{\lm}^{k}(X_{t})^{T}=H_{\lm}^{k}(X_{t})$,
which is pure since $T$ acts by the identity. So for $1<k<n$, in
the exact sequence 
\[
\to H^{k-1}(X_{0})\overset{\pi^{*}}{\to}H^{k-1}(\tilde{X}_{0})\overset{\imath_{\cE}^{*}}{\to}H^{k-1}(\cE)\overset{\tau}{\to}H^{k}(X_{0})\to,
\]
purity of $H^{k}(X_{0})$ $\implies$ $\tau=0$ $\implies$ $\imath_{\cE}^{*}$
surjective; therefore \begin{equation}\label{eq1.6e}
H^k(\tilde{X}_0)\twoheadrightarrow H^k(\cE)\;\;\;\text{for}\;\;\;\;\; k\leq n-2\;\; \text{and}\;\; k\geq n.
\end{equation}Of course, our assumption implies that $p$ is a hypersurface singularity,
so that $\QQ_{X_{0}}[d]$ is perverse; \eqref{eq1.6e} can then also
be derived from the resulting exact sequence $0\to\imath_{*}W\to\QQ_{X_{0}}[n]\to\IC_{X_{0}}\to0$
(in fact, we only need an isolated l.c.i. singularity here). The Clemens--Schmid and vanishing cycle
sequences' real strength is in using the smooth fibers' cohomology
to further constrain those of $X_{0}$ and $\cE$.

\section{Cyclic base-change and quotients}\label{sec-cyclic}

We return to a scenario analogous to \eqref{eq1.5c}, but where the
singularities need not be isolated. Begin with a flat projective family
$f:\cx'\to\Delta$ with $\Xi':=\text{sing}(X_{0}')\supseteq\text{sing}(\cx')$
($\implies$ $f^{-1}(\Delta^{*})=\cx'\backslash X_{0}'$ smooth),
and recall that for us $X_{0}'=\cup D_{i}$ is the \emph{reduction}
of the divisor $(f)=\sum\kappa_{i}D_{i}$, $\kappa_{i}\in\ZZ_{\geq0}$.
For lack of a less self-contradictory terminology, we shall say that
$\cx'$ has \emph{reduced special fiber} if all $\kappa_{i}=1$; this
implies in particular that $X_{0}'$ is Cartier.

Let $g:\cx''\to\Delta$ be a second family with a finite surjective
morphism $\rho:\cx''\to\cx'$ over a cyclic quotient $t\mapsto t^{\kappa}$;
and fix log resolutions $\cy',\cy''$ of $(\cx',X_{0}'),(\cx'',X_{0}'')$
to have a diagram \begin{equation}\label{eq1.7a}
\xymatrix{Y_0 ' \subset \cy' \ar [rd] \ar @{->>}[r]^{\pi'} & \cx' \ar [d]^f & \cx'' \ar [l]_{\rho} \ar[d]^g & \cy''\supset Y_0 '' \ar [ld] \ar @{->} [l]_{\pi''} \\ & \Delta & \Delta \ar [l]_{(\cdot)^{\kappa}} }
\end{equation}Writing $\cx$, $\cy$, etc. when we want to make a statement independent
of the decoration, we assume that the log resolutions $\pi$ are isomorphisms
off $X_{0}$ and write $Y_{0}=\tilde{X}_{0}\cup\cE$. Denote the monodromies
by $T':=T_{0}=T_{0}^{\text{ss}}e^{N_{0}}$ and $T''=T_{0}^{\kappa}$.

\subsection{Cyclic base-change \label{S1.7.1}}

An important special case of \eqref{eq1.7a} is where: \begin{itemize}[leftmargin=0.5cm]
\item $\rho$ is the base-change, so that $X_0 ' = X_0 '' =: X_0$;
\item $\cy''$ is the semi-stable reduction of $\cx'$ (so $\kappa$ must satisfy $(T_0^{\text{ss}})^{\kappa}=I$); and
\item $\cx'$ (hence $\cx''$) has reduced special divisor.
\end{itemize}In this case the SNCD $Y_{0}=\tilde{X}_{0}\cup\cE$, with $\tilde{X}_{0}\twoheadrightarrow X_{0}$
birational. When $\cx$ ($=\cx'$ or $\cx''$) is not smooth (though
we continue to assume $\cx\backslash X_{0}$ smooth), we would first
like to ``quantify'' the failure of the local invariant cycle theorem
for $\cx$.

Begin with the diagram of split short-exact sequences \begin{equation*}
\xymatrix@R-1pc@C-2pc{0 \ar [rd] & & & & & & 0 \ar [ld] \\ & \IH^k_{\pha}(\cx) \ar @/^2pc/ [rrrr] \ar [rd] & & H^k(X_0) \ar [ld]_{\widetilde{\gamma^*}} \ar [rd]^{\gamma^*} & & H^k_{\pha}(Y_0) \ar [ld] \\ && \IH^k(\cx) \ar @/^1pc/ [rr] \ar [rd] && H^k(Y_0) \ar [ld] \\ &&& H^k_{\lm}(Y_t)^T \ar [rd] \ar [ld] \\ && 0 & & 0}
\end{equation*}where $\Xi\overset{\alpha}{\leftarrow}\cE\overset{\beta}{\rightarrow}Y_{0}\overset{\gamma}{\rightarrow}X_{0}$
and $\widetilde{\gamma^{*}}$ is induced by $\QQ_{\cx}[n+1]\to\IC_{\cx}$.
By the decomposition theorem for $\cy\twoheadrightarrow\cx$, the
curved arrows are split injections as well. Clearly then
\[
H_{\pha}^{k}(X_{0}):=(\widetilde{\gamma^{*}})^{-1}\left(\IH_{\pha}^{k}(\cx)\right)=(\gamma_{*})^{-1}\left(H_{\pha}^{k}(Y_{0})\right)
\]
and\begin{equation} \label{eq1.7b}
\overline{H}^{k}(X_{0}):=H^{k}(X_{0})/H_{\pha}^{k}(X_{0})\cong\image\left\{ H^{k}(X_{0})\overset{\sp}{\to} H_{\lm}^{k}(X_{t})^{T}\right\} 
\end{equation}are independent of choices (of whether $\cx=\cx'$ or $\cx''$, and
of $\cy$).
\begin{rem}
\label{rem1.7} Since $\IH_{\pha}^{k}(\cx)=\IH^{k}(\cx)\cap H_{\pha}^{k}(Y_{0})\subset H^{k}(Y_{0})$,
\begin{equation*}
\begin{array}{ccccc}\IH^{k}(\cx) & \cong & \IH_{\pha}^{k}(\cx) & \oplus & H_{\lm}^{k}(X_{t})^{T}\\ \vinj &  & \vinj &  & \parallel \\ H^{k}(Y_{0}) & \cong & H_{\pha}^{k}(Y_{0}) & \oplus & H_{\lm}^{k}(X_{t})^{T}\end{array}
\end{equation*}exhibits the MHS $\IH^{k}(\cx)$ as a ``lower bound'' on the cohomology
of any resolution.
\end{rem}
On the other hand, the monodromy invariants $H_{\lm}^{k}(X_{t})^{T}$
are certainly not independent of the choice of $\cx$, and so the
cokernel of $\sp$ in \eqref{eq1.7b} cannot be. In view of the exact
sequence
\[
\to H^{k}(X_{0})\to H^{k}(Y_{0})\oplus H^{k}(\Xi)\overset{\beta^{*}-\alpha^{*}}{\longrightarrow}H^{k}(\cE)\overset{\delta}{\to}H^{k+1}(X_{0})\to
\]
we compute\begin{multline}\label{eq1.7c}
\mathrm{coker}\left( H^k(X_0)\to H^k_{\lm}(X_t)^T \right) = \mathrm{coker}\left( \overline{H}^k(X_0)\overset{\overline{\gamma^*}}{\to}\overline{H}^k(Y_0)\right)
\\
=\ker\left\{ \tfrac{H^k(\cE)}{\alpha^* H^k(\Xi)+\beta^* H^k_{\pha}(Y_0)} \overset{\delta}{\to}H^{k+1}(X_0)\right\} =: \overline{H}^k(\cE).
\end{multline}Accordingly, the replacement for Clemens--Schmid in this general context
becomes \begin{equation}\label{eq1.7d}
0\to \overline{H}^k(X_0)\to H^k_{\lm}(X_t)^T\to \overline{H}^k(\cE) \to 0.
\end{equation}

Specializing a bit more, suppose the pre-base-change family $\cx'$
is smooth: then $\IH^{k}(\cx')=H^{k}(X_{0})$ itself is the ``lower
bound'' for $H^{k}(Y_{0}')$, and $\overline{H}^{k}(\cE')=\{0\}$.
Since $\cy''\twoheadrightarrow\cx'$ is a proper morphism between
equidimensional smooth manifolds (cf. \cite{We}), we have ($H^{k}(\cx'')=H^{k}(X_{0})=$)
$H^{k}(\cx')\hookrightarrow H^{k}(\cy'')$ which yields\begin{equation}\label{1.7e}
\begin{array}{ccccc}H^{k}(X_{0}) & \hookrightarrow & \IH^{k}(\cx'') & \hookrightarrow & H^{k}(Y_{0}'')\\\parallel &  & \parallel &  & \parallel\\H_{\pha}^{k}(X_{0}) & \hookrightarrow & \IH_{\pha}^{k}(\cx'') & \hookrightarrow & H_{\pha}^{k}(Y_{0}'')\\\oplus &  & \oplus &  & \oplus\\H_{\lm}^{k}(X_{t})^{T_{0}} & \hookrightarrow & H_{\lm}^{k}(X_{t})^{T_{0}^{\kappa}} & = & H_{\lm}^{k}(X_{t})^{T_{0}^{\kappa}}\end{array}
\end{equation}(where $(\cdot)^{T_{0}^{\kappa}}=\ker(N_{0})$) and\begin{equation}\label{eq1.7f}
\overline{H}^{k}(\cE'')=H_{\lm}^{k}(X_{t})^{T_{0}^{\kappa}}/H_{\lm}^{k}(X_{t})^{T_{0}}.
\end{equation}So while (by Cor. \ref{cor1.4}) $\gr_{F}^{0}$ and $W_{k-1}$ of
$H^{k}(X_{0})$ and $H_{\lm}^{k}(X_{t})^{T_{0}}$ agree, this is false
in general if we replace $T_{0}$ by $T_{0}^{\kappa}$.

\subsection{Cyclic quotient singularities}

Rather than obtaining $\cx''$ from $\cx'$, we may wish to define
$\cx':=\cx''/G$, where $G=\langle g\rangle\cong\ZZ/\kappa\ZZ$ acts
nontrivially on $X_{0}$. In this case, $\cx'$ will essentially never
be smooth or have reduced special fiber. Nevertheless, it is \emph{always}
true (no need to assume $\cx''$ smooth; cf. \cite[Th. III.7.2]{Br})
that\begin{equation}\label{eq1.7g}
H^*(X_0')=H^*(\cx')\overset{\cong}{\underset{\rho^*}{\to}}H^*(\cx'')^G = H^*(X_0'')^G.
\end{equation}We have $g^{*}\in\mathrm{Aut}(H^{*}(X_{0}''))$, $T'\in\mathrm{Aut}(H_{\lm}^{*}(X_{t}'))$,
$T''\in\mathrm{Aut}(H_{\lm}^{*}(X_{t}''))$, and $H_{\lm}^{*}(X_{t}')=H_{\lm}^{*}(X_{t}'')=:H_{\lm}^{*}$.
One may perhaps know $g^{*}$ and $T''$, and wish to determine $T'$:
for instance, if the quotient has been used to form a singularity
of higher index (on $X_{0}'$) from one of index $1$ (on $X_{0}''$).
To that end we have the following
\begin{prop}
\label{prop1.7} $T'$ extends the action of $g^{*}$ on $\image\{H^{*}(X_{0}'')\to H_{\lm}^{*}\}$
$=:\bar{H}^{*}(X_{0}'')$ to all of $H_{\lm}^{*}$. In particular,
we have $\overline{H}^{*}(X_{0}')=(H_{\lm}^{*})^{T'}\cap\overline{H}^{*}(X_{0}'')$,
and so the local invariant cycle theorem holds for $\cx'$ if it holds
for $\cx''$.\end{prop}
\begin{proof}
If we analytically continue a basis $\{C_{i}\}\subset H_{*}(X_{t}'')$
to $H_{*}(X_{\zeta_{\kappa}t}'')$, then writing $g(C_{i})$ (call
this $\widetilde{g}_{*}$) in terms of these translates is just $g_{*}$
on the ``global'' cycles and corresponds to clockwise monodromy
``downstairs'' (in $t^{\kappa}$). The local-system monodromy $T'$
on cohomology is the \emph{transpose} of the latter (cf. \cite{DS}):
so $T'={}^{t}\widetilde{g}_{*}=\widetilde{g}^{*}$, and $\tilde{g}^{*}|_{\overline{H}^{*}(X_{0}'')}=g^{*}$.
\end{proof}
Now suppose $\cx''$ is smooth. Since Clemens--Schmid \eqref{eq1.4j}
holds for $\cx''$, taking $G$-invariant parts exactly gives\begin{multline}\label{eq1.7h}
0\to H^{k-2}_{\lm}(X_t ')_{T_0}(-1)\overset{\sp^{\vee}}{\to}H_{2n-k+2}(X_0')(-n-1) \\ \overset{\rho_* \circ\gy\circ \rho^*}{\longrightarrow} H^k(X_0')\overset{\sp}{\to}  H^k_{\lm}(X_t')^{T_0}\to 0 ,
\end{multline}since $((H_{\lm}^{*})^{T_{0}^{\kappa}})^{G}=(H_{\lm}^{*})^{T_{0}}$.
(As quotient singularities, those of $\cx'$ are rational \cite{KM},
but the results of $\S$\ref{S1.8} for rational singularities are
weaker than this.) Further, it is often possible to deduce $T'$ from
$g^{*}$ and $T''$ in this case. The action of $T''=(T_{0}^{\text{ss}})^{\kappa}e^{\kappa N_{0}}$
on $H_{\lm}^{*}$ extends to one of $\mathfrak{sl}_{2}\times\langle(T_{0}^{\text{ss}})^{\kappa}\rangle$,
compatibly with the Deligne bigrading $H_{\lm,\CC}^{*}=\oplus_{p,q}(H_{\lm}^{*})^{p,q}$.
Accordingly, it suffices to determine the choice of $\kappa^{\text{th}}$
root $T'$ of $T''$ on $\ker(N_{0})\subset H_{\lm}^{*}$. That of
$T''$ (resp. $T'$) decomposes $\ker(N_{0})\cong\oplus_{k\in\mathbb{N}}W_{k}^{\oplus m_{k}}$
(resp. $\oplus_{\ell\in\mathbb{N}}V_{\ell}^{\oplus n_{\ell}}$) over
$\QQ$, where $\det\{(\lambda I-(T_{0}^{\text{ss}})^{\kappa})|_{W_{k}}\}$
(resp. $\det\{(\lambda I-T_{0}^{\text{ss}})|_{V_{\ell}}\}$) is the
$k^{\text{th}}$ (resp. $\ell^{\text{th}}$) cyclotomic polynomial,
and so the issue is to compute the $\{n_{\ell}\}$ given $\{m_{k}\}$.
The point here is that since $\overline{H}^{*}(X_{0}'')=(H_{\lm}^{*})^{T''}=W_{1}^{\oplus m_{1}}$,
$g^{*}$ determines the $n_{\ell}$ for all $\ell|\kappa$, and one
can sometimes deduce the others from the formula $V_{\ell}=W_{\ell/(\ell,\kappa)}^{\oplus\{\phi(\ell)/\phi(\ell/(\ell,\kappa))\}}$.
For instance, if $\kappa|k$ then the only possibility is $W_{k}^{\oplus m_{k}}=V_{\kappa k}^{\oplus m_{k}/\kappa}$.
Conversely, this puts constraints on the set of $\ZZ/\kappa\ZZ$ by
which one can even consider taking cyclic quotients.

\subsection{Relative quotients}

A more general quotient scenario is where $G\overset{\theta}{\twoheadrightarrow}\ZZ/\kappa\ZZ$;
in $\S$1.7.2, $\theta$ was an isomorphism. Now we consider the opposite
extreme, where $\kappa=1$. More precisely, let $\cx'\overset{f}{\to}\Delta$
be flat, projective, and smooth over $\Delta^{*}$, with $X_{0}'=D_{a}\cup D_{b}$
\emph{generically} a reduced NCD along $D_{ab}=D_{a}\cup D_{b}$.
Suppose that each $x\in X_{0}'$ has a neighborhood $\cv$ arising
as a finite group quotient\[
\xymatrix@C=1em{\tilde{V}_a \cup \tilde{V}_b = s^{-1}(0) \ar @{^(->}[r] & \tilde{\cv}\ar @{->>} [rr]^{/G}_{\rho} \ar [rd]_s && \cv \ar [ld]^{t=f|_{\cv}} & t^{-1}(0)=V_a \cup V_b \ar @{_(->}[l] \\ && \Delta}
\]of a semistable degeneration $\tilde{\cv}/\Delta$. Then the vanishing
cycles of $f$ behave exactly as in a SSD:
\begin{prop}
\label{propA} In this situation, $\pphi_{f}\QQ_{\cx'}\cong\QQ_{D_{ab}}(-1)[-2]$.\end{prop}
\begin{proof}
Working locally, since $t$ is $G$-invariant\begin{flalign*}
{}^p\phi_t \QQ_{\cv} &= \,\pphi_t (R\rho_* \QQ_{\tilde{\cv}})^G = (\pphi_t R\rho_* \QQ_{\tilde{\cv}})^G = (R\rho_* \pphi_t \QQ_{\tilde{\cv}})^G \\ &= (R\rho_* \QQ_{\tilde{V}_{ab}}(-1)[-2])^G = \QQ_{V_{ab}}(-1)[-2],
\end{flalign*}where $V_{ab}$ denotes $V_{a}\cap V_{b}$ etc.\end{proof}
\begin{rem}
\label{remA} The above scenario arises frequently via weighted blow-ups:
typically one has\[
\xymatrix{\cx' \ar @/_2pc/ [dd]_{f} \ar @{^(->} [r] \ar [d]^b & \mathrm{Bl}_{\uw}(\aff^{n+2}) = \PP_{\tilde{\Sigma}_{\uw}} \ar [d]^{\beta} \\ \fx \ar [d] \ar @{^(->} [r] & \aff^{n+2}_{x_0,\ldots ,x_{n+1}} \ar [ld]^{x_0 \mapsto t}_{p} \\ \aff^1_t}
\]where $\fx$ has a singularity at $\underline{0}$ arising from cyclic
base-change, $\uw=(1,w_{1},\ldots,w_{n+1})$ and $\tilde{\Sigma}_{\uw}=\mathrm{fan}\{e_{0},e_{1},\ldots,e_{n+1},e_{0}+\sum_{i=1}^{n+1}w_{i}e_{i}\}$.
(In particular, Prop. \ref{propA} explains Example \ref{ex1.5c},
which will be generalized in Part II.) The exceptional divisor
of $\beta$ is $\mathbb{WP}(\uw)\cong\PP_{\Sigma_{\uw}}$, with $\Sigma_{\uw}=\mathrm{fan}\{e_{1},\ldots,e_{n+1},-\sum_{i=1}^{n+1}w_{i}e_{i}\}$
the image of $\tilde{\Sigma}_{\uw}$ under the projection $e_{0}\mapsto-\sum w_{i}e_{i}$;
the proper transform of $t^{-1}(0)$ is $\PP_{\Sigma_{0}}(\cong\aff^{n+1})$
where $\Sigma_{0}=\mathrm{fan}\{e_{1},\ldots,e_{n+1}\}$. Assuming
that $\cx'$ meets $\PP_{\Sigma_{\uw}}$ with multiplicity $1$ (so
$D_{a}=\cx'\cap\PP_{\Sigma_{\uw}}$ and $D_{b}=\cx'\cap\PP_{\Sigma_{0}}$),
it will suffice to exhibit $\PP_{\tilde{\Sigma}_{\uw}}$ locally as
a finite quotient of $\aff^{n+2}$ not branched along $\PP_{\Sigma_{\uw}}$
or $\PP_{\Sigma_{0}}$.

Writing $\sigma_{i}:=\RR_{>0}\langle e_{0},\ldots,\hat{e}_{i},\ldots,e_{n+1},e_{0}+\sum_{i=1}^{n+1}w_{i}e_{i}\rangle$
and $N_{i}:=\ZZ\langle e_{0},\ldots,\hat{e}_{i},\ldots,e_{n+1},e_{0}+\sum_{i=1}^{n+1}w_{i}e_{i}\rangle$,
set $U_{i}:=\mathrm{Spec}(\CC[\underline{x}^{\sigma_{i}\cap\ZZ^{n+2}}])$
and $\tilde{U}_{i}:=\mathrm{Spec}(\CC[\underline{x}^{\sigma_{i}\cap N_{i}}])\cong\aff_{u_{0},\ldots,u_{n+1}}^{n+2}$
(where $u_{i}=x_{i}^{\frac{1}{w_{i}}}$, and $u_{j}=x_{j}u_{i}^{-w_{j}}$
for $j\neq i$). The natural, generically $w_{i}:1$ morphism $\tilde{U}_{i}\overset{\tau_{i}}{\twoheadrightarrow}U_{i}$
is the quotient by $\ZZ/w_{i}\ZZ\ni1\mapsto\{u_{i}\mapsto\zeta_{w_{i}}u_{i},\, u_{j}\mapsto\zeta_{w_{i}}^{-w_{j}}u_{j}\,(j\neq i)\}$.
Since $w_{0}=1$ we get $t=x_{0}=\frac{x_{0}}{u_{i}}u_{i}=u_{0}u_{i}$
(as a function) so that $\tau_{i}$ is not branched along $\PP_{\Sigma_{\uw}}$
or $\PP_{\Sigma_{0}}$, and $s|_{\tilde{U}_{i}}=t|_{U_{i}}\circ\tau_{i}=p\circ\beta\circ\tau_{i}$
(as a mapping) is a SSD as desired. Note that this exhibits the singularity
type in $U_{i}$ (at the origin) as $\frac{1}{w_{i}}(1,-1,-w_{1},\ldots,\hat{w}_{i},\ldots,-w_{n+1})$;
while if the proper transform of $\cx'$ passes through a point with
$u_{k}=0$ for $k\neq j$, the quotient yields a point of type $\frac{1}{\gcd(w_{i},w_{j})}(1,-1,-w_{1},\ldots,\hat{w}_{i},\ldots,\hat{w}_{j},\ldots,-w_{n+1})$
on $\cx'$. \end{rem}
\begin{example}
So in Example \ref{ex1.5c}, $\uw=(1,6,14,21)$ yields points of type
$\frac{1}{2}(1,-1,1)$, $\frac{1}{3}(1,-1,1)$, $\frac{1}{7}(1,-1,1)$
on $\cx'$, which become $\frac{1}{2}(1,-1)$, $\frac{1}{3}(1,-1)$,
$\frac{1}{7}(1,-1)$ (i.e. $A_{1},A_{2},A_{6}$) on $D_{a}$ and $D_{b}$.
\end{example}

\section{Singularities of the minimal model program}\label{S1.8}

Let $X$ be a projective variety with resolution $\tilde{X}\overset{\ve}{\to}X$,
and extend this\footnote{Note that typically $\tilde{X}$ is only a connected component of $X^0$.} to a cubical hyperresolution $\eb:\xb\to X$ \cite{PS}.
\begin{defn}
(i) $X$ has \emph{rational} singularities $\iff$ $\ox\overset{\simeq}{\to}R\ve_{*}\co_{\tilde{X}}$.

(ii) $X$ has \emph{du Bois} singularities $\iff$ $\ox\overset{\simeq}{\to}R\eb_{*}\ob$.
\end{defn}
In general we have $\gr_{F}^{0}H^{k}(X)\cong\HH^{k}(X,R\eb_{*}\ob)$,
so that if $X$ is du Bois then $\gr_{F}^{0}H^{k}(X)\cong H^{k}(\ox)$.
This last isomorphism clearly also holds if $X$ has rational singularities:
since $\ox\to R\ve_{*}\co_{\tilde{X}}$ always factors through $R\eb_{*}\ob$,
we get a diagram\begin{equation*}
\xymatrix{&H^k(\co_{\tilde{X}})
\\
H^k(X,\CC) \ar @/^1pc/ [ru] \ar [r] \ar @{->>} @/_1pc/ [rd] & H^k(\ox) \ar [u]^{X\text{ has RS}\implies}_{\cong} \ar [d]_{(*)} \ar [r]^{(*)} & \HH^k(X,R\eb_* \co_{X_0^{\bullet}}) \ar @/_1pc/ [lu] \ar @{<->} @/^1pc/ [ld]^{\cong}
\\
& \gr_F^0 H^k(X,\CC)}
\end{equation*}forcing $(*)$ to be injective and surjective. With more work \cite{Kv},
one can show that rational singularities are in fact du Bois.

Now consider a flat projective family $f:\cx\to\Delta$, with $\cx\backslash X_{0}$
smooth and $\cy\overset{\pi}{\to}\cx$ a log-resolution of $(\cx,X_{0})$
(i.e. $\cy$ smooth, $Y_{0}=\pi^{-1}(X_{0})$ SNCD) restricting to
an isomorphism $\cy\backslash Y_{0}\cong\cx\backslash X_{0}$. We
shall assume that $\pi$ extends to a morphism $\bar{\pi}:\bar{\cy}\to\bar{\cx}$
of projective varieties ($\bar{\cy}$ smooth, $\cong$ off $X_{0}$), and that $f$ is also extendable to an algebraic morphism (from $\bar{\cx}$ to a curve).
\begin{prop}
\label{prop1.8a} If $\cx$ has rational singularities, then $\sp$
induces isomorphisms
\[
\gr_{F}^{0}H^{k}(X_{0})\cong\gr_{F}^{0}H_{\lm}^{k}(X_{t})^{T_{0}^{\text{ss}}}
\]
for all $k$.\end{prop}
\begin{proof}
We have $\gr_{F}^{0}H_{\lm}^{k}(X_{t})^{T_{0}^{\text{ss}}}\cong\gr_{F}^{0}H_{\lm}^{k}(X_{t})^{T_{0}}\overset{\cong}{\underset{\sp}{\leftarrow}}\gr_{F}^{0}H^{k}(Y_{0})$,
since $\cy$ is smooth and $\text{level}(H_{\pha}^{k}(Y_{0}))\leq k-2$.
Taking $\gr_{F}^{0}$ of the Mayer-Vietoris sequence (cf. \cite[Thm. 5.35]{PS})
\[
\to H^{k}(\bar{\cx})\to H^{k}(\bar{\cy})\oplus H^{k}(X_{0})\to H^{k}(Y_{0})\to
\]
and using that $\bar{\cx}$ is (rational $\implies$) du Bois yields
\[
\to H^{k}(\co_{\bar{\cx}})\to H^{k}(\co_{\bar{\cy}})\oplus\gr_{F}^{0}H^{k}(X_{0})\to\gr_{F}^{0}H^{k}(Y_{0})\to.
\]
Now $\co_{\bar{\cx}}\overset{\simeq}{\to}R\bar{\pi}_{*}\co_{\bar{\cy}}$
$\implies$ $H^{k}(\co_{\bar{\cx}})\overset{\cong}{\to}H^{k}(\co_{\bar{\cy}})$
$\implies$ $\gr_{F}^{0}H^{k}(X_{0})\cong\gr_{F}^{0}H^{k}(Y_{0})$
gives the result.
\end{proof}
We next make use of an ``inversion of adjunction'' result of Schwede
\cite{Sw}, that when a Cartier divisor (with smooth complement) is
du Bois, the ambient variety has only rational singularities. However,
this requires us to place an additional constraint on $\cx$ to ensure
that $X_{0}$ is Cartier and remains so after base-change.
\begin{thm}
\label{prop1.8b} If $X_{0}$ has du Bois singularities, and $\cx$
has reduced special fiber \emph{(}i.e. $(f)=X_{0}$\emph{)}, then
\begin{equation}\label{eq1.8a}
\gr_F^0 H^k(X_0)\cong \gr_F^0 H^k_{\lm} (X_t)^{T_0^{\text{ss}}} \cong \gr_F^0 H_{\lm}^k (X_t)\;\;\;(\forall k).
\end{equation}\end{thm}
\begin{proof}
Obviously $\cx$ has rational singularities by {[}op. cit.{]}, but
this also applies to any finite base-change. So taking $\cx'=\cx$
in the setting of $\S$\ref{S1.7.1}, Prop. \ref{prop1.8a} applies
in addition to $\cx''$, whose $T''=T_{0}^{\kappa}$ is unipotent.\end{proof}
\begin{rem}
If $\cx$ is smooth, then (regardless of whether $(f)=X_{0}$) one
can show that $X_{0}$ is du Bois iff $\co_{X_{0}}\overset{\simeq}{\to}R\pi_{*}^{0}\co_{Y_{0}}$
\cite{Sw}.\end{rem}
\begin{example}
To see the necessity of the $(f)=X_{0}$ requirement in Thm. \ref{prop1.8b},
consider a smooth $\cx$ with elliptic fibers over $\Delta^{*}$ and
$X_{0}$ a Kodaira type $\mathrm{IV}^{*}$ (``$E_{6}$'') fiber.
Then $(f)=3D_{1}+2(D_{2}+D_{3})+(D_{4}+D_{5}+D_{6})\neq\sum_{i}D_{i}=X_{0}$;
and sure enough, the conclusion of Thm. \ref{prop1.8b} fails (cf.
Example \ref{ex1.5a}).
\end{example}

\begin{example}
Assume $\Sigma=\{p\}$ is an isolated quasi-homogeneous singularity
of type $(f\sim)\, F=x^{2}+y^{3}+z^{6}+\lambda xyz$ ($\tilde{E}_{8}$)
resp. $G=x^{5}+y^{5}+z^{2}$ ($N_{16}$). In the first case, $X_{0}$
is du Bois. As we can see from Example \ref{ex1.5b}, the discrepancy
between $H^{2}(X_{0})\cong H_{\lm}^{2}(X_{t})^{T_{0}}$ and $H_{\lm}^{2}(X_{t})^{T_{0}^{6}}=\ker(N_{0})$
consists of $8$ $(1,1)$ classes, and neither differ from $H_{\lm}^{2}(X_{t})$
on $\gr_{F}^{0}$. Any base-change $F+t^{M}$ defines a rational $3$-fold
singularity, since (as one deduces from the absence of integral interior
points in the convex hull of $\{(2,0,0),\,(0,3,0),\,(0,0,6),\,(0,0,0)\}$)
the exceptional divisor $\cE$ of the weighted blow-up has $\Omega^{2}(\cE)=\{0\}$.

In the second case (where $T_0=T_0^{\text{ss}}$), $H^{2}(\cE)$ has Hodge numbers $(1,14,1)$, so
that $\gr_{F}^{0}$ of $H^{2}(X_{0})\cong H_{\lm}^{2}(X_{t})^{T_{0}}$
and $H_{\lm}^{2}(X_{t})$ differ by $1$. The point is that (while
$\cx$ is smooth) $X_{0}$ is not du Bois and neither is (say) $G+t^{10}$;
so in particular, $\cx''$ will not have rational singularities.
\end{example}
Returning to our resolution $\tilde{X}\overset{\ve}{\to}X$, assume
now%
\footnote{Serre's condition $S_{2}$ is ``algebraic Hartogs'': given any $Z\overset{\jmath}{\subset}X$
of codim.$\geq2$, $\jmath_{*}\co_{X\backslash Z}=\co_{X}$; so it
easily follows that normality is equivalent to $\ve_{*}\co_{\tilde{X}}=\ox$.%
}\begin{itemize}[leftmargin=0.5cm]
\item $X$ is normal (smooth in codim. 1, and satisfies $S_2$)
\item $X$ is $\QQ$-Gorenstein ($K_X$ is $\QQ$-Cartier)
\end{itemize}and write $K_{\tilde{X}}=\ve^{*}K_{X}+\sum_{i}m_{i}E_{i}$ ($E_{i}$
exceptional prime divisors).
\begin{defn}
$X$ has \emph{terminal} (resp. \emph{canonical}, \emph{log-terminal},
\emph{log-canonical}) singularities $\iff$ all $m_{i}$ are $>0$
(resp. $\geq0$, $>-1$, $\geq-1$). 
\end{defn}
A larger class of singularities is obtained by dropping the ``smooth
in codimension 1'' part of normality:
\begin{defn}
Assume $X$ satisfies $S_{2}$ and is $\QQ$-Gorenstein, and has only
normal-crossing singularities in codimension 1. Let $\hat{X}\to X$
be the normalization and $\hat{D}$ the conductor (inverse image of
the normal-crossing locus); let $\hat{Y}\overset{\hat{\pi}}{\to}\hat{X}$
be a log-resolution of $(\hat{X},\hat{D})$. Then $X$ has \emph{semi-log-canonical
(slc)} singularities $\iff$ the $m_{i}\geq-1$ in $K_{\hat{Y}}+\hat{\pi}_{*}^{-1}\hat{D}=\hat{\pi}^{*}(K_{\hat{X}}+\hat{D})+\sum_{i}m_{i}\hat{E}_{i}$
($\hat{E}_{i}$ exceptional).
\end{defn}
We have two related ``inversion of adjunction'' results here \cite{KM,Ka}:
if a Cartier divisor with smooth complement in a normal, $\QQ$-Gorenstein
variety is log-terminal (resp. slc), then the ambient variety has
only terminal (resp. canonical) singularities. In addition, we know
that log-terminal (resp. slc) singularities are rational (resp. du
Bois) \cite{Kv,KK,Ko}. Thus we arrive at the following
\begin{cor}
\label{cor1.8a} Assume our family $f:\cx\to\Delta$ has normal, $\QQ$-Gorenstein
total space, and reduced special fiber $X_{0}$.

\emph{(i)} If $X_{0}$ is slc, then \eqref{eq1.8a} holds.

\emph{(ii)} If $X_{0}$ is log-terminal, \eqref{eq1.8a} holds and
$W_{k-1}\gr_{F}^{0}H_{\lm}^{k}(X_{t})=\{0\}$.\end{cor}
\begin{proof}
For (ii), $\gr_{F}^{0}H^{k}(X_{0})\cong H^{k}(\co_{X_{0}})$ (since
$X_{0}$ is du Bois) and $\gr_{F}^{0}H^{k}(\tilde{X}_{0})\cong H^{k}(\co_{\tilde{X}_{0}})$,
where $H^{k}(\tilde{X}_{0})$ is pure of weight $k$. Since $X_{0}$
has rational singularities, $H^{k}(\co_{X_{0}})\cong H^{k}(\co_{\tilde{X}_{0}})$.
\end{proof}
We can think of (i) in terms of Hodge-Deligne numbers as saying that
\[
h^{k}(X_{0})^{p,q}=h_{\lm}^{k}(X_{t})^{p,q}\;\;\text{for}\; p\cdot q=0,
\]
 and (ii) as saying that moreover both are zero for $(p,q)=(r,0)$
or $(0,r)$ with $r\neq k$.

In the log-terminal case, we have \begin{equation} \label{eq1.8b}
\gr_{F}^{1}H_{\lm}^{k}(X_{t})=\gr_{F}^{1}H_{\lm}^{k}(X_{t})^{T^{\text{un}}}
\end{equation}by (ii), which might kindle hopes that perhaps this equals $\gr_{F}^{1}H^{k}(X_{0})$.
Unfortunately, nothing quite this strong is true at any level of generality
one can specify in terms of the singularity types described above:
for $n=3$, the nicest such scenario would be where $\cx$ is smooth
and $X_{0}$ has Gorenstein terminal ($\iff$ isolated compound du
Val) singularities.
\begin{example}
Such a singularity is given locally by $f\sim x^{2}+y^{2}+zw^{4}+z^{2}w^{2}+z^{4}w$,
whose contribution to $H_{\lm}^{3}(X_{0})$ has nontrivial $(1,1)$
and $(1,2)$ parts, with neither part $T^{\text{ss}}$-invariant hence
neither appearing in $H^{3}(X_{0})$. This assertion will be justified
in Part II.
\end{example}
In any case, here is something one \emph{can} say:
\begin{thm}
\label{cor1.8b} If $X_{0}$ is log-terminal \emph{(}or more generally,
has rational singularities\emph{)},\footnote{We emphasize that we do \emph{not} assume isolated singularities here.} and $\cx$ is smooth, then $\gr_{F}^{1}H^{k}(X_{0})\cong\gr_{F}^{1}(H_{\lm}^{k}(X_{t}))^{T^{\text{ss}}}$.\end{thm}
\begin{proof}
Begin by observing that under the duality functor $\DD$ on MHM we have $\DD \QQ_{\cx}[n+1]=\QQ_{\cx}[n+1](n+1)$ hence (on $X_0$)
\begin{equation}\label{eqf1}
\DD \pphi_f^u \QQ_{\cx}[n+1]=\pphi_f^u \QQ_{\cx}[n+1](n+1)
\end{equation}
by \cite[(2.6.2)]{mhm}.  Since $\HH^{j-n}(X_0,\pphi_f^u \QQ_{\cx}[n+1])=H^j_{\van}(X_t)^{T^{\text{ss}}}$ by definition, and
\begin{multline*}
\HH^{j-n}(X_0,\DD\,\pphi_f^u \QQ_{\cx}[n+1])\cong \HH^{n-j}(X_0,\pphi_f^u \QQ_{\cx}[n+1])^{\vee}\\=(H^{2n-j}_{\van}(X_t)^{T^{\text{ss}}})^{\vee},
\end{multline*}
taking the direct image of \eqref{eqf1} by $X_0\to \mathrm{pt}$ induces ($\forall j\in \ZZ$) a perfect pairing
\begin{equation}\label{eqf2}
H^j_{\van}(X_t)^{T^{\text{ss}}}\times H^{2n-j}_{\van}(X_t)^{T^{\text{ss}}}\to \QQ(-n-1).	
\end{equation}
In particular, $F^n H^{2n-j}_{\van}(X_t)^{T^{\text{ss}}}$ is dual to $H^j_{\van}(X_t)^{T^{\text{ss}}}/F^2$ in \eqref{eqf2}.

Now recall that $X_0$ has (not necessarily isolated) rational singularities, and $\cx$ is smooth.  By a result of M. Saito \cite[Thms. 0.4-0.6]{Sa4}, we therefore have 
\begin{equation}\label{eqf3}
\left( F^n (\pphi_f^u \QQ_{\cx}[n+1]) \subseteq \right) \, F^n (\pphi_f \QQ_{\cx}[n+1])=\{0\}	,
\end{equation}
where $\QQ_{\cx}[n+1]$ is interpreted as a MHM and $F^n$ is $F_{-n}$ in \emph{op. cit.} Again taking $H^{n-j}$ of the direct image, we conclude ($\forall j\in\ZZ$) that $F^n H^{2n-j}_{\van}(X_t)^{T^{\text{ss}}} =\{0\}$ hence $H^j_{\van}(X_t)^{T^{\text{ss}}}=F^2 H^j_{\van}(X_t)^{T^{\text{ss}}}$.  Taking $\gr_F^1$ of the $T^{\text{ss}}$-invariant part of the vanishing cycle sequence \eqref{eqI2}, the result follows.
\end{proof}

\begin{rem}
\textbf{(i)} The main issue dealt with in Thm. \ref{cor1.8b} is the vanishing of the $(1,k-1)$ part of $H^k_{\pha}(X_0)$.  Indeed, one reduces to this statement as follows: 
taking $T^{\text{ss}}$-invariants of \eqref{eq1.8b} gives $\gr^1_F H^k_{\lm}(X_t )^{T^{\text{ss}}} \cong \gr^1_F H^k_{\lm}(X_t )^T$; while applying C-S for $\cx$ smooth (Theorem \ref{thMain}) yields $\gr^1_F H^k_{\lm}(X_t )^T \cong \sp (\gr^1_F H^k(X_0))$, and $\ker\{\sp: H^k(X_0)\to H^k_{\lm}(X_t)\} =\delta(H_{\van}^{k-1}(X_t))$ has pure weight $k$ by Proposition \ref{prop1.4}.

\textbf{(ii)} According to M. Saito \cite{Sa-letter}, for $X_0$ du Bois [resp. rational], the conclusions of Thm. \ref{prop1.8b} and Cor. \ref{cor1.8a}(i) [resp. Cor. \ref{cor1.8a}(ii) and Thm. \ref{cor1.8b}] hold if we assume $\cx$ is smooth K\"ahler, $f$ is proper, and $X_0$ is reduced --- in particular, one need not assume that $f$ extends to an algebraic morphism (or that $\cx$ extends to a projective variety).

\textbf{(iii)} The result of Thm. \ref{prop1.8b} also holds for $\cx$ a complex analytic space (neither smooth nor extendable-to-algebraic) provided we assume that $\cx\setminus X_0$ is smooth, $f$ is \emph{projective} and $X_0$ is a reduced and irreducible divisor with \emph{rational} singularities \cite{KLS}.
\end{rem}

One can say quite a bit more with the aid of spectra, especially in the case of isolated singularities.  For example, in Part II we will show that when $\cx$ is smooth and $X_0$ has isolated $k$-log-canonical singularities in the sense of \cite{MP}, one has $\gr_F^j H^k (X_0) \cong \gr_F^j H^k_{\lm}(X_t)$ for $j=0,\ldots,k$.

\section{Decomposition Theorem over a polydisk}\label{S1.9}

We conclude by elaborating the consequences of Theorem \ref{th1.2} for the simplest multiparameter setting of all.  Let $f:\cx\to\Delta^r$ be a projective map of relative dimension $n$, equisingular\footnote{More precisely, we assume that the restrictions of $R^i f_* \IC_{\cx}\QQ$ to $U_{\{I\}}$ are local systems ($\forall i,I$).} over $(\Delta^*)^r$ and each ``coordinate $(\Delta^*)^k$''; and take $\mathcal{K}^{\bullet}=\IC_{\cx}$.  For notation we shall use:
\begin{itemize}[leftmargin=0.5cm]
\item $s_1,\ldots,s_r$ for the disk coordinates;
\item $U_{\{I\}}$ [resp. $\cs_{\{I\}}$] for the coordinate $(\Delta^*)^{r-|I|}$ [resp. $\Delta^{r-|I|}$] where $s_i = 0$ ($\forall i\in I$);
\item $\Delta^r \underset{\imath_c}{\hookleftarrow} \cs_c :=\cup_{|I|=c}\cs_{\{I\}} \underset{\jmath_c}{\hookleftarrow} \cup_{|I|=c}U_{\{I\}}=:U_c$; and
\item $\cup_{|I|=c}\{f_{\{I\}}^{\circ}\}:=f_c^{\circ}:\,\cx_c^{\circ} \rightarrow U_c$ resp. $\cx_c \underset{f_c}{\rightarrow} \cs_c$ for restrictions of $f$. 
\end{itemize}
In \eqref{eq1.2a}, $Z_{d(=r-c)}$ is replaced by $\cs_c$, and the pure weight-$(j+n+c)$ VHS $\VV_{r-c}^j$ over $U_c$ is rewritten $\mathsf{H}_c^{j+n+c}$ (restricting to $\mathsf{H}^{j+n+c}_{\{I\}}$ on each $U_{\{I\}}$ with $|I|=c$). For $c=0$, we write $\ch^{j+n}:= \mathsf{H}_0^{j+n}$, with fibers $\IH^{j+n}(X_{\underline{s}})$ ($s\in(\Delta^*)^r$) and monodromies $T_1,\ldots,T_r$.  For $c=1$, the $\mathsf{H}_{\{i\}}^{j+n+1}$ are the phantom $\IH$'s of fibers of $f_{\{i\}}^{\circ}$.

\begin{rem}
When $\cx$ is smooth, $\mathcal{K}^{\bullet}=\QQ_{\cx}[n+r]$, and the $\mathsf{H}^{j+n+c}_c$ are (pure) sub-VMHS of $R^{j+n+c}(f^{\circ}_c)_*\QQ_{\cx_c^{\circ}}$, with $\ch^{j+n}\cong R^{j+n}(f_0^{\circ})_* \QQ_{\cx_0^{\circ}}$.	
\end{rem}

With this indexing by codimension, the terms of \eqref{eq1.2a} become 
\begin{equation}\label{eq1.9.1}
\pr^j f_* \IC_{\cx} \simeq \oplus_c \imath^c_* \jmath_{!*}^c \mathsf{H}_c^{j+n+c}[r-c],	
\end{equation}
so that
\begin{equation}\label{eq1.9.2}
\begin{aligned}
\IH^m(\cx) &\cong \HH^{m-n-r}(\Delta^r,R f_* \IC_{\cx})\\ 
&\cong \oplus_{j,c} \IH^{m-n-c-j}(\cs_c,\mathsf{H}_c^{j+n+c})\\
&\cong	\oplus_{c,\ell} \IH^{\ell}(\cs_c,\mathsf{H}_c^{m-\ell})
\end{aligned}
\end{equation}
where $\ell:=m-(j+n+c)$.  There are two things to note here:  first, that\footnote{For simplicity, we write this as $\IH^{\ell} (\mathsf{H}^{m-\ell}_c)_{\uo}$ below; this notation means the \emph{stalk} cohomology $H^{\ell-r+c}\imath_{\uo}^* \imath^c_* \jmath^c_{!*} (\mathsf{H}^{m-\ell}_c [r-c])$, \emph{not} the (costalk) cohomology with \emph{support} at $\uo$.} $\IH^{\ell}(\cs_c,\mathsf{H}^{m-\ell}_c)=\oplus_{|I|=c}\IH^{\ell}(\cs_{\{I\}},\mathsf{H}_{\{I\}}^{m-\ell})=\oplus_{|I|=c}\IH^{\ell}(\mathsf{H}_{\{I\}}^{m-\ell})_{\uo}$ are really just sums of local $\IH$ groups at $\uo$. These are naturally endowed with \emph{mixed} Hodge structures by setting $\mathsf{H}_{\{I\},\lm}^* := (\prod_{j\notin I} \psi_{s_j})\mathsf{H}^*_{\{I\}}$ (or just $H_{\lm}^* :=\psi_{s_1}\cdots \psi_{s_r}\ch^*$ for $c=0$) and defining Koszul complexes $\mathscr{K}^{\bullet}(\mathsf{H}_{\{I\}}^*)$ by 
\begin{equation}\label{eq1.9.3}
\mathsf{H}^*_{\{I\},\lm} \to \oplus_{j\notin I}N_j \mathsf{H}_{\{I\},\lm}^*(-1) \to \oplus_{j_1 <j_2 \notin I}N_{j_1}N_{j_2}\mathsf{H}_{\{I\},\lm}^*(-2)\to \cdots\; ;
\end{equation}
then (as MHSs)
\begin{equation}\label{eq1.9.4}
\IH^{\ell}(\mathsf{H}_{\{I\}}^*)_{\uo} \cong H^{\ell}(\mathscr{K}^{\bullet}(\mathsf{H}_{\{I\}}^*))^{G_I}	
\end{equation}
where $G_I$ is the (finite) group generated by the $\{T^{\text{ss}}_j \}_{j\notin I}$ \cite{MR890924,CKS}.  We shall write $\mathsf{H}_{\{I\},\text{inv}}^*:=\IH^0 (\mathsf{H}^*_{\{I\}})_{\uo}$ for the $\{T_j\}_{j\notin I}$-invariants in $\mathsf{H}_{\{I\},\lm}^*$, and $\mathsf{H}_{c,\text{inv}}^* =\oplus_{|I|=c}\mathsf{H}^*_{\{I\},\text{inv}}$.  Obviously, \eqref{eq1.9.4} vanishes for $\ell>\max\{r-|I|-1,0\}$.

Second, the hard Lefschetz isomorphisms \eqref{eq1.2b} take the form $$\mathsf{H}^{-j+n+c}_c (-j) \overset{\cong}{\to} \mathsf{H}^{j+n+c}_c ,$$ so that $\mathsf{H}_c^*$ is centered about $*=n+c$. Since it is zero for $*>2n$, it is also zero for $*<2c$. Taking stock of these vanishings, \eqref{eq1.9.2} becomes 
\begin{equation}\label{eq1.9.5}
\IH^m(\cx)\cong \oplus_{c=0}^{\min(r,\lfloor\frac{m}{2}\rfloor)} \oplus_{\ell=0}^{\max(0,r-c-1)} \IH^{\ell}(\cs_c,\mathsf{H}^{m-\ell}_c).
\end{equation}

Now we introduce two filtrations:  the \emph{coniveau filtration} (by codimension of support) is just
\begin{equation} \label{eq1.9.6}
\mathscr{N}^{\alpha}\IH^m(\cx) :=\oplus_{c=\alpha}^{\min(r,\lfloor\frac{m}{2}\rfloor)} \oplus_{\ell=0}^{\max(0,r-c-1)} \IH^{\ell}(\cs_c,\mathsf{H}_c^{m-\ell}) ;	
\end{equation}
while the \emph{shifted perverse Leray filtration} $\mathscr{L}^{\alpha}:=\pl^{\alpha-r}$ is given (cf. \eqref{1.2d}-\eqref{1.2e}) by 
\begin{equation}\label{eq1.9.7}
\mathscr{L}^{\alpha}\IH^m(\cx):=\oplus_{c=0}^{\min(r,\lfloor\frac{m}{2}\rfloor)}\oplus_{\ell=\max(\alpha-c,0)}^{\max(r-c-1,0)} \IH^{\ell}(\cs_c,\mathsf{H}_c^{m-\ell}).	
\end{equation}
The following is essentially a special case of \cite{dCM3}:
\begin{prop}\label{propS9}
$\mathscr{L}^{\alpha}$ is the kernel of restriction to $f^{-1}(\mathcal{A})$, where $\mathcal{A} \subset \Delta^r$ is a general affine\footnote{More precisely, we mean the intersection of $\alpha-1$ hypersurfaces of the form $\mathsf{L}(\underline{s})=\mathsf{K}$, where $\mathsf{L}$ is a linear form and $\mathsf{K}$ a sufficiently small nonzero constant.} slice of codimension $\alpha-1$; and $\mathscr{N}^{\bullet}\subseteq \mathscr{L}^{\bullet}$.
\end{prop}
\begin{proof}
The restrictions 
\begin{equation} \label{eq1.9.8}
\IH^{\ell}(\cs_c,\mathsf{H}_c^{m-\ell})\to \IH^{\ell}(\cs_c \cap \mathcal{A}^{\alpha-1},\mathsf{H}^{m-\ell}_c)	
\end{equation}
are either injective or zero.  Here the target is computed by a \v Cech-Koszul double complex, which has the $N_{j_1}\cdots N_{j_{\ell}}\mathsf{H}_{c,\lm}^{m-\ell}(-\ell)$ terms required if and only if $\cs_{c+\ell}\cap \mathcal{A}^{\alpha-1}$ is nonempty. So \eqref{eq1.9.8} is zero $\iff$ $\ell\geq \alpha-c$ (as required).

The inclusion $\mathscr{N}^{\alpha} \subseteq \mathscr{L}^{\alpha}$ is now geometrically obvious, though it also follows directly from \eqref{eq1.9.6}-\eqref{eq1.9.7} by $c\geq \alpha \implies \alpha-c\geq 0 \implies \max(\alpha-c,0)\geq 0$.
\end{proof}
Taken together, these filtrations endow every term in the double sum \eqref{eq1.9.5} with geometric meaning:  from \eqref{eq1.9.6} and \eqref{eq1.9.7} we have $\gr_{\mathscr{L}}^{\alpha}=\oplus_{\ell=0}^{\alpha} \IH^{\ell}(\mathsf{H}^{m-\ell}_{\alpha-\ell})_{\uo}$ and $\gr_{\mathscr{N}}^{\beta}\IH^m(\cx)=\oplus_{\ell=0}^{\max(r-\beta-1,0)}\IH^{\ell}(\mathsf{H}^{m-\ell}_{\beta})_{\uo}$, whereupon
\begin{equation}\label{eq1.9.9}
\IH^{\ell}(\mathsf{H}_c^{m-\ell})_{\uo}=\gr_{\mathscr{L}}^{\ell+c}\gr_{\mathscr{N}}^c \IH^m(\cx).	
\end{equation}
Recalling that $\cx_1 = f^{-1}(\{s_1\cdots s_r=0\})$, we also obtain a generalization of (part of) the Clemens--Schmid sequence:
\begin{thm}\label{thmS9}
The sequence of MHS
\begin{equation}\label{eq1.9.10}
\IH_{\cx_1}^m(\cx)\longrightarrow \frac{\IH^m(\cx)}{\oplus_{\ell=1}^{r-1}\IH^{\ell}(\ch^{m-\ell})_{\uo}}\overset{\sp}{\longrightarrow}\IH^m_{\lm}(X_{\underline{s}})^{T_1,\ldots,T_r}\longrightarrow 0 \end{equation}
is exact.	
\end{thm}
\begin{proof}
Actually more is true:  $\IH^m(\cx)$ is the \emph{direct sum} of $\mathscr{N}^1 \IH^m(\cx)=\text{image}\{\IH^m_{\cx_1}(\cx)\to \IH^m(\cx)\}$ and $\gr_{\mathscr{N}}^0 \IH^m(\cx) = \oplus_{\ell=0}^{r-1}\IH^{r-1}(\ch^{m-\ell})_{\uo}$, with $\IH^0(\ch^m)_{\uo}=H_{\text{inv}}^m=(H_{\lm}^m)^{T_1,\ldots,T_r}$.
\end{proof}
For the remainder of the section, we \emph{assume that $\cx$ is smooth}, so that \eqref{eq1.9.10} becomes\footnote{Alternatively one can move the denominator of the middle term to the right-hand term as a direct summand.}
\begin{equation}\label{eq1.9.11}
H_{\cx_1}^m(\cx)\overset{\imath_{X_{\uo}}^* \imath_*^{\cx_1}}{\longrightarrow} \frac{H^m(X_{\uo})}{\oplus_{\ell=1}^{r-1}\IH^{\ell}(\ch^{m-\ell})_{\uo}}\overset{\sp}{\longrightarrow}H^m_{\text{inv}}(X_{\underline{s}})\longrightarrow 0 .
\end{equation}
(Note that we are \emph{not} assuming unipotent monodromies.) It is instructive to write out the decomposition \eqref{eq1.9.5} in detail for small $r$:
\begin{itemize}[leftmargin=0.5cm]
\item ($r=1$) $H^m(X_0)\cong H^m_{\text{inv}} \oplus \underset{\mathscr{L}^1}{\underbrace{\mathsf{H}_1^m}}$
\item ($r=2$) $H^m(X_{\uo})\cong H^m_{\text{inv}} \oplus \underset{\mathscr{L}^1}{\underbrace{\IH^1(\ch^{m-1})_{\uo}\oplus \mathsf{H}_{1,\text{inv}}^m\oplus \overset{\mathscr{L}^2}{\overbrace{\mathsf{H}_2^m}}}}$
\item ($r=3$) $H^m(X_{\uo})\cong \\ H^m_{\text{inv}}\oplus\underset{\mathscr{L}^1}{\underbrace{\left\{\begin{array}{c} \IH^1 (\ch^{m-1})_{\uo}\\ \oplus \mathsf{H}_{1,\text{inv}}^m \end{array}\right\} \oplus \overset{\mathscr{L}^2}{\overbrace{\left\{ \begin{array}{c} \IH^2 (\ch^{m-2})_{\uo}\oplus \\ \IH^1(\mathsf{H}_1^{m-1})_{\uo} \oplus \mathsf{H}_{2,\text{inv}}^m  \end{array} \right\} \oplus \underset{\mathscr{L}^3}{\underbrace{\mathsf{H}_3^m}}}}}}$
\end{itemize}
in which $\mathsf{H}_c^* =0$ for $*\leq 2c-1$ (and $\mathsf{H}^*_1$ is the phantom cohomology in codimension 1).  By Prop. \ref{propS9}, $\mathscr{L}^1$ is the kernel of the restriction to a nearby fiber $X_{\underline{s}}$ (i.e. of $\sp$), $\mathscr{L}^2$ of the restriction to a nearby affine line (meeting all coordinate hyperplanes), and so on.

Finally, here are a few examples which illustrate the $r=2$ scenario (and which all happen to have unipotent monodromy):

\begin{example}\label{ex1.9a}
Let $\mathcal{C}\to \Delta^2$ be a family of curves with smooth total space.  Then $H^1(C_{\uo})=H^1_{\text{inv}}$ and $H^2(C_{\uo})=H^2_{\text{inv}}(\cong \QQ(-1))\oplus \IH^1_{\uo}(\ch^1)\oplus \mathsf{H}^2_{1,\text{inv}}$. The simplest example with $\IH^1$-term nonzero is when $\mathcal{C}$ is a family of elliptic curves with $I_1$-fibers on $\{0\}\times \Delta^* \cup \Delta^*\times\{0\}$ (with equal monodromies $N_1 = N_2$) and $I_2$-fiber at $\uo$ (cf. \cite{MR2796415}); then $\mathsf{H}^2_1 =0$ and $H^2(C_{\uo})\cong \QQ(-1)^{\oplus 2}$.  For instance, if we base-change a 1-variable $I_1$ degeneration by $(s_1,s_2)\mapsto s_1 s_2$, the nonvanishing of $\IH^1(\ch^1)$ simply indicates that without blowing up, we have a singular total space.
\end{example}

\begin{example}\label{ex1.9b}
Abramovich and Karu \cite{AK} defined a notion of semistable degenerations in 	more than one parameter; these are characterized by having (i) smooth total space (so that \eqref{eq1.9.11} applies) and (ii) local structure of a fiber product of SSDs along the coordinate hyperplanes.  (In particular, they have unipotent monodromies.)  An easy case is that of an ``exterior product'' of 1-variable SSDs:  for instance, let $\cE
\overset{g}{\to}\Delta$ be a semistable degeneration of elliptic curves with $I_k$ singular fiber, so that $\cx:=\cE\times\cE \overset{g\times g}{\longrightarrow}\Delta \times \Delta$ has fibers $E_{s_1}\times E_{s_2}$. (Locally this takes the form $s_1=xy,\,s_2=zw$.)  

Regardless of $k$, we have $\IH^1(\ch^*)=0$.  For $k=1$, the $\mathsf{H}_{\{i\}}$ and $\mathsf{H}_{\{12\}}^*$ all vanish, so that $H^m(X_{\uo})=H^m_{\text{inv}}$ in all degrees.  However, when $k=2$, we have $\mathsf{H}^*_{\{i\},\text{inv}}=H^{*-2}_{\lm}(E_s)(-1)$, and $\mathsf{H}^*_{\{12\}}=H^{*-4}(\text{pt.})(-2)$; the reader may check that $H^m_{\text{inv}}\oplus \mathsf{H}^m_{\{1\},\text{inv}}\oplus \mathsf{H}^m_{\{2\},\text{inv}}\oplus \mathsf{H}^m_{\{12\}}$ correctly computes $H^m(E_0\times E_0)$. 
\end{example}

\begin{example}\label{ex1.9c}
Mirror symmetry allows for the computation of (unipotent) monodromies $T_i = e^{N_i}$ of families of CY toric hypersurfaces $X_{\underline{s}}\subset \PP$ in the ``large complex structure limit''.  In particular, \cite[$\S$8.3]{KPR} and \cite{Grimmetal} study two distinct 2-parameter families of $h^{2,1}=2$ CY 3-folds over $(\Delta^*)^2$ with Hodge-Tate LMHS $H^3_{\lm}$ at the origin.  The notation $\langle \mathrm{IV}_1 \mid \mathrm{IV}_2 \mid \mathrm{III}_0 \rangle$ for the first family and $\langle \mathrm{III}_0 \mid \mathrm{IV}_2 \mid \mathrm{III}_0 \rangle$ for the second indicates the LMHS types corresponding to $N_1$ (on $\{0\}\times \Delta^*$), $N_1+N_2$ (at $\{\uo\}$), and $N_2$ (on $\Delta^* \times \{0\}$). These types are described by their Hodge-Deligne diagrams: \[\includegraphics[scale=0.7]{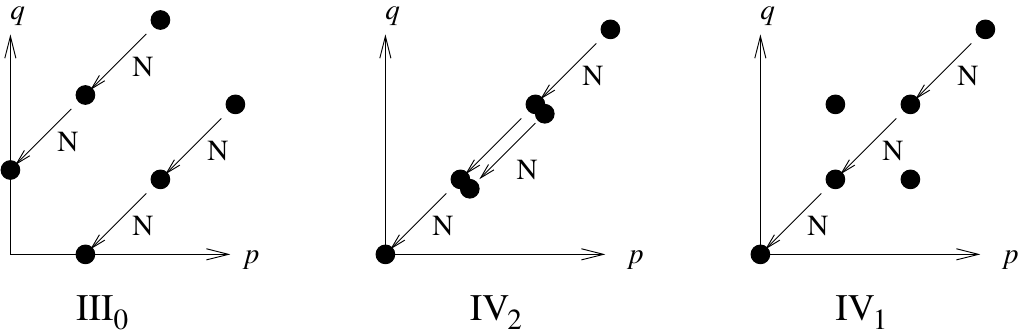}\]
Both variations have $H^3_{\text{inv}}\cong \QQ(0)$ and $H^4_{\text{inv}}\cong H^2(\PP)(-1)$ for their respective toric varieties, and of course $\IH^1 (\ch^*)_{\uo} =0$ for $*\neq 3$.  
Let us assume we have smooth compactifications of both families with all $\mathsf{H}^*_{1,\text{inv}}$ and $\mathsf{H}^*_2$ zero.  
Then $H^3(X_{\uo})=\QQ(0)$ and $H^4(X_{\uo})=H^2(\PP)(-1)\oplus \IH^1 (\ch^3)_{\uo}$ in both cases; so the key to the topology of $X_{\uo}$ in each case (provided we want a smooth total space) lies in the cohomology of the complex $H_{\lm}^3 \to N_1 H_{\lm}^3(-1) \oplus N_2 H_{\lm}^3(-1) \to N_1 N_2 H^3_{\lm}(-2)$.  From the LMHS types one immediately deduces that (writing ranks of maps over the arrows) this complex takes the form $\CC^6 \overset{5}{\to} \CC^3 \oplus \CC^4 \overset{2}{\to} \CC^2$ in the first case (so that $\IH^1(\ch^3)=0$), and $\CC^6 \overset{5}{\to}\CC^4 \oplus \CC^4 \overset{2}{\to} \CC^2$ in the second (so that $\IH^1(\ch^3)=\QQ(-2)$).
\end{example}

\newcommand{\etalchar}[1]{$^{#1}$}
\providecommand{\bysame}{\leavevmode\hbox to3em{\hrulefill}\thinspace}
\providecommand{\MR}{\relax\ifhmode\unskip\space\fi MR }
\providecommand{\MRhref}[2]{%
  \href{http://www.ams.org/mathscinet-getitem?mr=#1}{#2}
}
\providecommand{\href}[2]{#2}
\bigskip

\appendix

\pagestyle{myheadings}
\markboth{}{}

\section*{Appendix : Decomposition Theorem over Stein Curves}
\ms\centerline{Morihiko Saito}
\pagestyle{myheadings}
\markboth{APPENDIX : DECOMPOSITION THEOREM OVER STEIN CURVES}{MORIHIKO SAITO}
\ms\bsn
In this Appendix we prove the following (see Theorem A.4 below):
\msn
{\bf Theorem A.} {\it Let $f:X\to C$ be a proper surjective morphism of a connected complex manifold $X$ to a connected non-compact curve $C$. The decomposition theorem for $\Rb f_*\Qb_X$ is equivalent to the Clemens-Schmid exact sequence $($or the local invariant cycle theorem$)$ for every singular fiber of $f$.}
\ms
We also show that the weak decomposition always holds if $C$ is a connected non-compact smooth curve, see Corollary A.3 below.
Note that the last hypothesis implies that $C$ is {\it Stein{\hskip1pt}} by Behnke-Stein, see for instance \cite[Corollary 26.8]{F}.
We have the following.
\msn
{\bf Corollary~A.} {\it For $f:X\to C$ as above, the decomposition theorem holds for $\Rb f_*\Qb_X$, if there is an embedded resolution $\rho:\Xt\to X$ such that the inverse image $\Xt_c:=\rho^{-1}(X_c)$ of any singular fiber $X_c:=f^{-1}(c)$ is a divisor with simple normal crossings $($not necessarily reduced\,$)$ and there is a cohomology class $\eta_c\in H^2(\Xt_c,\Cb)$ whose restriction to any irreducible component of $\Xt_c$ is a K\"ahler class.}
\ms
Note that the local invariant cycle theorem holds under the above hypotheses as is well-known, see also Remark~A.4\,(ii) below.
\sk
This work is partially supported by JSPS Kakenhi 15K04816.
\msn
{\bf A.1.~$t$-structure on complex manifolds} (see \cite{BBD}).
Let $X$ be a complex manifold, and $A$ be any subfield of $\Cb$. Let $D^b_c(X,A)$ be the bounded derived category of $A$-complexes with constructible cohomology sheaves. For $k\in\Zb$, we have the full subcategories
$$D^b_c(X,A)^{\les k}\,\subset\,D^b_c(X,A),$$
defined by the following condition for $\Ks\in D^b_c(X,A):$
$$\dim{\rm Supp}\,\Hc^j\Ks\les k-j\q\q(\forall\,j\in\Zb).
\leqno{\rm(A.1.1)}$$
Put
$$\aligned D^b_c(X,A)^{\ges k}&:=\Db^{-1}\bl(D^b_c(X,A)^{\les-k}\br)\\&\,\,\bl(=\bl\{\Ks\in D^b_c(X,A)\mid\Db \Ks\in D^b_c(X,A)^{\les-k}\br\}\br),\\
D^b_c(X,A)^{[k]}&:=D^b_c(X,A)^{\les k}\,\cap\,D^b_c(X,A)^{\ges k}.\endaligned
\leqno{\rm(A.1.2)}$$
Here, taking an injective resolution $A_X\simto\Ic^{\ssb}$, we can define $\Db \Ks$ by
$$\Db \Ks=\tau_{\les k}\Hc om_A(\Ks,\Ic^{\ssb}(\dim X)[2\dim X])\q(k\gg 0),
\leqno{\rm(A.1.3)}$$
where $\tau_{\les k}$ is a classical truncation.
\sk
By \cite{BBD}, the $D^b_c(X,A)^{[k]}$ are {\it abelian{\hskip1pt}} full subcategories of $D^b_c(X,A)$, and there are truncation functors
$${}^p\tau_{\les k}:D^b_c(X,A)\to D^b_c(X,A)^{\les k},$$
(similarly for ${}^p\tau_{\ges k}$) together with the cohomological functors
$$^p\Hc^k:D^b_c(X,A)\to D^b_c(X,A)^{[0]},$$
and also the distinguished triangles for $\Ks\in D^b_c(X,A):$
$${}^p\tau_{\les k-1}\Ks\to{}^p\tau_{\les k}\Ks\to({}^p\Hc^k\Ks)[-k]\buildrel{+1}\over\to.
\leqno{\rm(A.1.4)}$$
\msn
{\bf A.2.~Curve case.} Assume $\dim X=1$, that is, $X$ is a smooth curve $C$. It is well-known that $D^b_c(C,A)^{[k]}\subset D^b_c(C,A)$ is defined by the following conditions:
$$\aligned\dim{\rm Supp}\,\Hc^k\Ks&=0,\\\Hc_{\{c\}}^0\Hc^{k-1}\Ks&=0\q(\forall\,c\in C),\\ \Hc^j\Ks&=0\q(\forall\,j\notin\{k, k-1\}).\endaligned
\leqno{\rm(A.2.1)}$$
This can be shown using the functor $i_c^!$ for $i_c:\{c\}\into C$ together with duality.
\sk
The following proposition and lemma are also well-known:
\msn
{\bf Proposition~A.2.} {\it For any $\Ks\in D^b_c(C,A)^{[k]}$, there is a unique finite increasing filtration $G$ on $\Ks$ satisfying
$$\aligned\Gr^G_0\Ks&=(j_*L)[1{-}k],\\ \dim{\rm Supp}\,\Gr^G_i\Ks&=0\q\h{if}\,\,\,\,|i|=1,\\ \Gr^G_i\Ks&=0\q\h{if}\,\,\,\,|i|>1.\endaligned
\leqno{\rm(A.2.2)}$$
where $L$ is an $A$-local system on a Zariski-open subset $C'\buildrel{j'}\over\into C$ which is obtained by restricting $\Hc^{k-1}\Ks$ to $C'$ $($note that $|C\setminus C'|$ may be infinite$)$. Moreover $\Gr^G_{-1}\Ks$ $($resp. $\Gr^G_1\Ks)$ is the maximal subobject $($resp. quotient object$)$ of $\Ks$ supported on a discrete subset of $C$.}
\msn
{\it Proof.} Set
$$G_1\Ks:=\Ks,\q G_0\Ks:=\tau_{\les k-1}\Ks,$$
where $\tau_{\les k-1}$ is the truncation in the classical sense. By (A.2.1) we have the canonical isomorphisms
$$G_0\Ks=(\Hc^{k-1}\Ks)[1{-}k],\q\Gr^G_1\Ks=(\Hc^k\Ks)[-k],
\leqno{\rm(A.2.3)}$$
together with the short exact sequence of sheaves
$$0\to\Hc^{k-1}\Ks\buildrel{\iota}\over\to j_*L\to{\rm Coker}\,\iota\to 0,
\leqno{\rm(A.2.4)}$$
inducing a short exact sequence in $D^b_c(C,A)^{[k]}:$
$$0\to({\rm Coker}\,\iota)[-k]\to(\Hc^{k-1}\Ks)[1{-}k]\to(j_*L)[1{-}k]\to 0.
\leqno{\rm(A.2.5)}$$
So we get (A.2.2), setting $G_{-1}\Ks:=({\rm Coker}\,\iota)[-k]$, $G_{-2}\Ks:=0$. The last assertion follows from Lemma~A.2 below.
This finishes the proof of Proposition~A.2.
\msn
{\bf Lemma~A.2.} {\it In the notation of Proposition~{\rm A.2}, the shifted direct image sheaf $(j_*L)[1]$ is canonically isomorphic to the intermediate direct image $j_{!*}(L[1])$ in $D^b_c(C,A)^{[0]}$ $($see \cite{BBD}$)$.}
\msn
{\it Proof.} We have the following short exact sequences in $D^b_c(C,A)^{[0]}$\,:
$$\aligned&0\to({\rm Coker}\,\iota')\to(j_!L)[1]\to(j_*L)[1]\to 0,\\ &0\to(j_*L)[1]\to(\Rb j_*L)[1]\to R^1j_*L\to 0,
\endaligned
\leqno{\rm(A.2.6)}$$
where $\iota':j_!L\into j_*L$ is a canonical inclusion.
(These two short exact sequences are dual of each other if $L$ in the second sequence is replaced by its dual.)
Lemma~A.2 then follows.
\ms
Lemma~A.2 and Proposition~A.2 imply the following.
\msn
{\bf Corollary~A.2.} {\it Any simple object of $D^b_c(C,A)^{[k]}$ is either $A_c[-k]$ with $A_c$ a sheaf supported at a point $c\in C$ or $(j_*L)[1{-}k]$ with $L$ a simple $A$-local system on a Zariski-open subset $C'\buildrel{j}\over\into C$.}
\ms
(Note, however, that the intermediate direct image $j_{!*}$ and the direct image $j_*$ are {\it not{\hskip1pt}} exact functors.)
\msn
{\bf Remark~A.2.} The intermediated direct image $j_{!*}(L[1])$ is also written as ${\rm IC}_CL$, and is called the {\it intersection complex{\hskip1pt}} (with local system coefficients).
\msn
{\bf A.3.~Vanishing of higher extension groups.} In the case of non-compact curves, we have the following.
\msn
{\bf Proposition~A.3.} {\it If $C$ is a connected non-compact smooth curve, and $\Ks,\Kps\in D^b_c(C,A)^{[k]}$, we have}
$${\rm Ext}_{D^b(C,A)}^i(\Ks,\Kps)=0\q(i\ges 2).
\leqno{\rm(A.3.1)}$$
\msn
{\it Proof.} Since the assertion is independent of $k\in\Zb$, we may assume $k=1$. Set
$$E^{\ssb}:=\Rb\Hc om_A(\Ks,\Kps)\in D^b_c(C,A).$$
We first reduce the assertion (A.3.1) to
$$E^{\ssb}\in D^b_c(C,A)^{\les 1},
\leqno{\rm(A.3.2)}$$
using the well-known isomorphism
$$\aligned{\rm Ext}_{D^b(C,A)}^i(\Ks,\Kps)&=H^i\bl(C,\Rb\Hc om_A(\Ks,\Kps)\br)\\&=H^i(C,E^{\ssb}),\endaligned
\leqno{\rm(A.3.3)}$$
together with the Riemann-Hilbert correspondence and also Cartan's Theorem~B.
(Here some finiteness condition would be needed if we use the duality for the direct images of objects of $D^b_c(C,A)$ by the morphism $C\to{\rm pt}$.)
\sk
By scalar extension $A\into\Cb$ we may assume $A=\Cb$. Let $M$ be the regular holonomic left $\Dc_C$-module $M$ corresponding to $^p\Hc^jE^{\ssb}$ ($j\les 1$). It has a global good filtration $F$, and we have the following quasi-isomorphism for $k\gg 0$\,:
$$C\bl({\rm d}:F_kM\to F_{k+1}M\otimes_{\Oc_C}\Omega_C^1\br)\simto{\rm DR}(M)={}^p\Hc^jE^{\ssb}.$$
Recall that any connected {\it non-compact{\hskip1pt}} smooth curve is {\it Stein{\hskip1pt}} as a consequence of the theory of Behnke-Stein, see for instance \cite[Corollary 26.8]{F}. We then get by Cartan's Theorem~B
$$H^i(C,{}^p\Hc^jE^{\ssb})=0\q(i>0).$$
Here one problem is that it is not quite clear whether $k$ exists {\it globally,} since $C$ is non-compact. For this we can use Proposition~A.2 so that the assertion is reduced to the intersection complex case. Then the Deligne extension \cite{D2} gives the filtration $F$ with the above $k$ rather explicitly.
(Note that Cartan's Theorem~B does not necessarily hold for quasi-coherent sheaves, see for instance \cite[Remark~2.3.8\,(2)]{mhp}.)
The assertion (A.3.1) is thus reduced to (A.3.2).
\sk
Let $C'\subset C$ be a Zariski-open subset such that $\Ks|_{C'}$, $\Kps|_{C'}$ are local systems. Since the assertion (A.3.2) is local, we may assume that $(C,C')=(\De,\De^*)$ so that $E^{\ssb}|_{\De^*}$ is a local system. The assertion is then reduced to that
$$(\Hc^jE^{\ssb})_0={\rm Ext}^j_{D^b(\De,\Cb)}(\Ks,\Kps)=0\q\q(j\ges 2).
\leqno{\rm(A.3.4)}$$
\sk
Using the Riemann-Hilbert correspondence, the latter assertion is equivalent to that
$${\rm Ext}^j_{\Dc_{\De,0}}(M,M')=0\q\q(j\ges 2),
\leqno{\rm(A.3.5)}$$
for any regular holonomic $\Dc_{\De,0}$-modules $M,M'$.
This is further reduced to the case where $M,M'$ are simple regular holonomic $\Dc_{\De,0}$-modules (using the standard long exact sequences of extension groups). So we may assume that $M,M'$ are of the form $\Dc_{\De,0}/\Dc_{\De,0}P$ with
$$P\,=\,\dd_t,\,\,\,t,\,\,\,t\dd_t-\alpha\,\,\,(\alpha\in\Cb\setminus\Zb),
\leqno{\rm(A.3.6)}$$
where $t$ is a coordinate of $\De$.
This implies a free resolution
$$0\to\Dc_{\De,0}\buildrel{\cdot\,P}\over\longrightarrow\Dc_{\De,0}\to M\to 0,$$
and shows (A.3.5). (Here it is also possible to use the isomorphism $\Rb\Hc om_A(\Ks,\Kps)=\delta^!(\Db\Ks\boxtimes\Kps)$ with $\delta:X\into X\times X$ the diagonal, although the argument is more complicated.)
This finishes the proof of Proposition~A.3.
\ms
From Proposition~A.3 we can deduce the following.
\msn
{\bf Theorem~A.3} (Weak Decomposition theorem). {\it Let $C$ be a connected non-compact smooth curve. For any $\Ks\in D^b_c(C,A)$, we have a non-canonical isomorphism}
$$\Ks\cong\mopl_{j\in\Zb}\,({}^p\Hc^j\Ks)[-j]\q\h{in}\,\,\,\,D^b_c(C,A),
\leqno{\rm(A.3.7)}$$
\msn
{\it Proof.} This follows from Proposition~A.3 by using the distinguished triangles in (A.1.4) by induction on $k$.
\msn
{\bf Corollary~A.3.} {\it Let $f:X\to C$ be a proper morphism of complex manifolds with $C$ a connected non-compact smooth curve. Let $X'\subset X$ be a Zariski-open subset. Set $f':=f|_{X'}$, $\,^p\!R^jf'_*:={}^p\Hc^j\Rb f'_*$, and $d_X:=\dim X$.
Then we have a non-canonical isomorphism}
$$\Rb f'_*(A_{X'}[d_X])\cong\mopl_{j\in\Zb}\,\bl({}^p\!R^jf'_*(A_{X'}[d_X])\br)[-j]\q\h{in}\,\,\,\,D^b_c(C,A).
\leqno{\rm(A.3.8)}$$
\msn
{\bf Remarks~A.3.} (i) It is quite unclear whether $\Rb f'_*A_{X'}$ belongs to $D^b_c(C,A)$ unless we assume that $f':X'\to C$ can be extended to a proper morphism  of complex manifolds $f:X\to C$ with $X\setminus X'$ a closed analytic subset of $X$.
\ms
(ii) If $f:X\to Y$ is a {\it smooth{\hskip1pt}} projective morphism of complex manifolds, we have the {\it weak decomposition} (see \cite{D1}):
$$\Rb f_*(A_X[d_X])\cong\mopl_{j\in\Zb}\,\bl({}^p\!R^jf_*(A_X[d_X])\br)[-j]\q\h{in}\,\,\,\,D^b_c(Y,A),
\leqno{\rm(A.3.9)}$$
using the Leray spectral sequence together with the hard Lefschetz property
$$\ell^j:{}^p\!R^{-j}f_*(A_X[d_X])\simto{}^p\!R^jf_*(A_X[d_X])(j)\q(j>0),
\leqno{\rm(A.3.10)}$$
since ${}^p\!R^jf_*(A_X[d_X])=(R^{d_X-d_Y}f_*A_X)[d_Y]$ in the $f$ smooth case.
\ms
(iii) In the $f$ non-smooth case, we need a ``spectral object'' in the sense of Verdier \cite{Ve} in order to extend the above argument, see also \cite[Lemma~5.2.8]{mhp}. Note that the proof of the decomposition theorem in \cite[Theorem 6.2.5]{BBD} is {\it completely different{\hskip1pt}} from this. It uses mod $p$ reduction, and the coefficients are $\Cb$, not $\Qb$.
\ms
(iv) The decomposition theorem in the derived category of mixed Hodge modules was {\it not{\hskip1pt}} proved in \cite{mhp} (it follows from \cite[(4.5.4)]{mhm}).
We can deduce from \cite{mhp} only the decompositions of the underlying $\Qb$-complex and the underlying complex of filtered $\Dc$-modules together with some compatibility between the decomposition isomorphisms, using Deligne's argument on the ``uniqueness'' in \cite{D4}. We have to use \cite[Lemma 5.2.8 and Proposition~2.1.12]{mhp} to apply Deligne's argument respectively to the $\Qb$-complex and the complex of filtered $\Dc$-modules (by passing from the derived category of filtered $\Dc$-modules to that of the abelian category of graded ${\mathcal B}$-modules), see also \cite[2.4--5]{toh}.
\msn
{\bf A.4.~Clemens-Schmid sequence.} For $\Ks\in D^b_c(\De,A)$, we have the following diagram for the {\it octahedral axiom{\hskip1pt}} of derived categories (see also \cite[Remark 5.2.2]{mhp})\,:
$$\begin{array}{cccclcccccl}
i_0^*\Ks&&\!\!\!\!\longleftarrow&&\!\!\!\!i_0^!\Ks&&\!\!\!\!i_0^*\Ks&&\!\!\!\!\longleftarrow&&\!\!\!\!i_0^!\Ks\\
&\!\!\!\!\!\!\!\!\nwarrow&\!\!\!\!{\scriptstyle c}&\!\!\!\!\!\!\!\!\swarrow&&&&\!\!\!\!\!\!\!\!\searrow&\!\!\!\!{\scriptstyle d}&\!\!\!\!\nearrow\!\!\!{\scriptstyle+1}\\
\,\,\,\downarrow\!\!{\scriptstyle+1}&\!\!\!\!\!\!\!{\scriptstyle d}\,\,\,&\!\!\!\!{}^p\varphi_1\Ks&\!\!\!\!{\scriptstyle d}&\uparrow\!\!{\scriptstyle+1}&&
\,\,\,\downarrow\!\!{\scriptstyle+1}&\!\!\!\!\!\!\!{\scriptstyle c}\,\,\,&\!\!\!\!C(N)&\!\!\!\!{\scriptstyle c}&\uparrow\!\!{\scriptstyle+1}\\
&\!\!\!\!\!\!\!\!\raise1mm\h{$\scriptstyle\rm can$}\!\!\!\!\nearrow\,\,&\!\!\!\!{\scriptstyle c}&\!\!\!\!\!\!\!\!\searrow\!\!\!\!\!\raise1mm\h{$\scriptstyle\rm Var$}&&&&\!\!\!\!\swarrow\!\!\!{\scriptstyle+1}&\!\!\!\!{\scriptstyle d}&\!\!\!\!\!\!\!\!\nwarrow\\
{}^p\psi_1\Ks&&\!\!\!\!\buildrel{N}\over\longrightarrow&&\!\!\!\!\!\!\!\!\!\!{}^p\psi_1\Ks(-1)&&\!\!\!\!{}^p\psi_1\Ks&&\!\!\!\!\buildrel{N}\over\longrightarrow&&\!\!\!\!\!\!\!\!\!\!{}^p\psi_1\Ks(-1)\\
\end{array}$$
Here $c$ and $d$ mean respectively commutative and distinguished, and $i_0:\{0\}\into\De$ is the inclusion. We denote respectively by ${}^p\psi_1$, $^p\varphi_1$ the {\it unipotent{\hskip1pt}} monodromy part of the shifted nearby and vanishing cycle functors ${}^p\psi:=\psi[-1]$, ${}^p\varphi:=\varphi[-1]$ for the coordinate $t$ of $\De$.
\sk
In the above diagram, the following two distinguished triangles are respectively called the {\it vanishing cycle triangle} (see \cite{D3}) and the {\it dual vanishing cycle triangle}\,:
$$\aligned{}^p\psi_1\Ks&\buildrel{\rm can\,\,}\over\longrightarrow{}^p\varphi_1\Ks\to i_0^*\Ks\buildrel{+1}\over\to,\\ i_0^!\Ks&\to{}^p\varphi_1\Ks\buildrel{\rm Var\,\,}\over\longrightarrow{}^p\psi_1\Ks(-1)\buildrel{+1}\over\to.\endaligned
\leqno{\rm(A.4.1)}$$
These are dual of each other if $\Ks$ in the second triangle is replaced by $\Db\Ks$, see for instance \cite[Lemma~5.2.4]{mhp}.
\sk
The {\it Clemens-Schmid sequence{\hskip1pt}} is associated to the outermost part of the above diagram as follows:
$$\aligned&\to H^{j-1}i_0^!\Ks\to H^{j-1}i_0^*\Ks\to H^j{\hskip1pt}^p\psi_1\Ks\buildrel{N}\over\to H^j{\hskip1pt}^p\psi_1\Ks(-1)\\&\to H^{j+1}i_0^!\Ks\to H^{j+1}i_0^*\Ks\to\cdots\endaligned.
\leqno{\rm(A.4.2)}$$
There are two sequences depending on the parity of $j\in\Zb$, see also \cite{Cl}.
Note that this sequence is essentially {\it self-dual,} more precisely, its dual sequence is isomorphic to the Clemens-Schmid sequence for $\Db\Ks$. This follows from the duality between the two distinguished triangles in (A.4.1).
\sk
For $\Ks\in D^b_c(\De,A)$, we say that the {\it Clemens-Schmid exact sequence holds{\hskip1pt}} if the above two sequences are {\it exact{\hskip1pt}} at every term.
\sk
We say that the  {\it local invariant cycle property holds{\hskip1pt}} if the above sequence is exact at the third term, that is, if we have the exactness of
$$H^ji_0^*\Ks\to H^j\psi_1\Ks\buildrel{N}\over\to H^j\psi_1\Ks(-1)\q(\forall\,j\in\Zb).
\leqno{\rm(A.4.3)}$$
\sk
We say that the {\it strong decomposition holds{\hskip1pt}} for $\Ks\in D^b_c(C,A)$ with $C$ a smooth curve if there is a non-canonical isomorphism
$$\Ks\cong\mopl_{j\in\Zb}\,{\rm IC}_CL^j[-j]\,\oplus\,\mopl_{c\in C,j\in\Zb}\, E_c^j[-j]\q\h{in}\,\,\,D^b_c(C,A),
\leqno{\rm(A.4.4)}$$
where the $L^j$ are local systems defined on a Zariski-open subset of $C$, and the $E_c^j$ are sheaves (in the classical sense) supported at $c\in C$.
\sk
If the weak decomposition holds (that is, if the isomorphism (A.3.7) holds), then the strong decomposition is equivalent to the following {\it canonical{\hskip1pt}} isomorphisms called the {\it cohomological decompositions}\,:
$${}^p\Hc^j\Ks={\rm IC}_CL^j\,\oplus\,\mopl_{c\in C}\, E_c^j\q\h{in}\,\,\,D^b_c(C,A)^{[0]}\q(\forall\,j\in\Zb).
\leqno{\rm(A.4.5)}$$
These isomorphisms are {\it canonical{\hskip1pt}} by {\it strict support decomposition} (see \cite[5.1.3]{mhp}), and (A.4.5) is equivalent to the following direct sum decompositions at every $c\in C$\,:
$${\rm Im}\,{\rm can}\,\oplus\,{\rm Ker}\,{\rm Var}={}^p\varphi_{t,1}{}^p\Hc^j\Ks\q(\forall\,j\in\Zb),
\leqno{\rm(A.4.6)}$$
where $t$ is a local coordinate of $(C,c)$.
\sk
By Theorem~A.3, the weak decomposition holds for any complex $\Ks\in D^b_c(C,A)$ if $C$ is connected and non-compact.
\sk
We have the following.
\msn
{\bf Theorem A.4.} {\it Let $C$ be a connected non-compact smooth curve. Let $\Ks\in D^b_c(C,A)$ with a self-duality isomorphism $\Db\Ks\cong\Ks[m]$ for some $m\in\Zb$. Then the following three conditions are equivalent to each other\,$:$
\skn
{\rm (a)} The strong decomposition holds.
\skn
{\rm (b)} The Clemens-Schmid exact sequence holds at any $c\in C$.
\skn
{\rm (c)} The local invariant cycle property holds at any $c\in C$.}
\msn
{\it Proof.} We first prove (a) $\Rightarrow$ (b). Restricting to each direct factor of $\Ks$, it is enough to consider the case $\Ks=j_*L$ with $L$ a local system, where the self-duality assumption $\Db\Ks\cong\Ks[m]$ is forgotten for the moment. (Indeed, in the case $\Ks=E_c$, we can use the functorial isomorphisms $i_c^!\ssc(i_c)_*=i_c^*\ssc(i_c)_*={\rm id}$ with $i_c:\{c\}\into C$ the inclusion.) We have to show the exact sequence
$$0\to i_0^*j_*L\to \psi_1L\buildrel{N}\over\to\psi_1L(-1)\to H^2i_0^!j_*L\to 0,
\leqno{\rm(A.4.7)}$$
where $\psi_1j_*L$ is denoted by $\psi_1L$.
We have the exactness at the second term by definition. This implies the exactness at the third term, if we remember the self-duality condition $\Db\Ks\cong\Ks[m]$ (which implies that $\Db(j_*L)\cong(j_*L')[2]$ for some direct factor $(j_*L')[m']\subset\Ks$) together with the self-duality of the Clemens-Schmid sequence explained after (A.4.2). So the implication (a) $\Rightarrow$ (b) follows.
\sk
Assume now condition (c) (since the implication (b) $\Rightarrow$ (c) is trivial). Let $\De\subset C$ be a sufficiently small neighborhood of $c$ with coordinate $t$. Condition~(c) implies that
$${\rm Ker}\,{\rm can}={\rm Ker}\,N\q\h{in}\,\,\,{}^p\psi_{t,1}{}^p\Hc^j\Ks\q(j\in\Zb),$$
using the long exact sequence associated to the vanishing cycle triangle in (A.4.1), since
$${\rm Var}\ssc{\rm can}=N\q\h{on}\q{}^p\psi_{t,1}{}^p\Hc^j\Ks.$$
Considering the coimages and images of can and $N$, it induces the isomorphism
$${\rm Var}:{\rm Im}\,{\rm can}\simto{\rm Im}\,N,\q\h{hence}$$
$${\rm Im}\,{\rm can}\cap{\rm Ker}\,{\rm Var}=0.
\leqno{\rm(A.4.8)}$$
We then get the direct sum decomposition (A.4.6) using the self-duality isomorphism
$$\Db\,{}^p\Hc^j\Ks={}^p\Hc^{-j}\Db\Ks\cong{}^p\Hc^{m-j}\Ks\q(j\in\Zb),$$
since can and Var are dual of each other (up to a sign) as is explained after (A.4.1), see also Remark~A.4\,(i) below. So condition~(a) follows.
This finishes the proof of Theorem~A.4.
\msn
{\bf Remarks~A.4.} (i) It is well-known that any indecomposable regular holonomic $\Dc_0$-module $M$ (with $\Dc_0:=\Dc_{\De,0}$) is isomorphic to one of the following:
$$\begin{array}{clll}
\h{(A)}&\Dc_0/\Dc_0(\dd_tt)^i&(K=\Rb j_*j^*K)&i\ges 1,\\
\h{(B)}&\Dc_0/\Dc_0(t\dd_t)^i&(K=j_!j^*K)&i\ges 1,\\
\h{(C)}&\Dc_0/\Dc_0(\dd_tt)^{i-1}\dd_t&(K=j_*j^*K)&i\ges 1,\\
\h{(D)}&\Dc_0/\Dc_0t(\dd_tt)^i&&i\ges 0,\\
\h{(E)}&\Dc_0/\Dc_0(\dd_tt{-}\alpha)^i&(K=j_!j^*K=\Rb j_*j^*K)&i\ges 1,\\\end{array}$$
where $\alpha\in\Cb\setminus\Zb$, $\,K:={\rm DR}(M)[-1]\in D^b_c(C,A)^{[1]}$, and $i$ is the rank of the local system $K|_{\De^*}$. Note that their duals are respectively (B), (A), (C), (D), (E) (with $\alpha$ changed).
We can prove the assertion, for instance, calculating the extension groups of simple regular holonomic $\Dc_0$-modules in (A.3.6).
The local invariant cycle property implies that indecomposable $\Dc_0$-modules of type (A), (C), (D) (with $i=0$), (E) are allowed, and the self-duality excludes the type (A). Here we use the short exact sequences
$$0\to\Dc_0/\Dc_0P\buildrel{\cdot\,Q\,\,}\over\longrightarrow\Dc_0/\Dc_0PQ\to\Dc_0/\Dc_0Q\to0.$$
This implies that $\Oc_{\De,0}$ is a $\Dc_0$-submodule of an indecomposable regular holonomic $\Dc_0$-module $M$ only for type (A), (C). (This classification argument is used in a detailed version of \cite{HF2}.)
\ms
(ii) The local invariant cycle theorem holds for a proper morphism of complex manifolds $f:X\to\De$ if there is an embedded resolution $\rho:(\Xt,\Xt_0)\to(X,X_0)$ such that $\Xt_0$ is a divisor with simple normal crossings (not necessarily reduced) and there is a cohomology class $\eta_c\in H^2(\Xt_0,\Cb)$ whose restriction to any irreducible component of $\Xt_c$ is represented by a K\"ahler form. This follows for instance from the arguments in \cite[4.2.2 and 4.2.4]{mhp} (see also arXiv:math/0006162). It is known that the argument in \cite{St} is insufficient, see for instance \cite{GN} where the singular fiber is assumed reduced. (It does not seem very clear whether one can prove the semi-stable reduction theorem in the analytic case using the same argument as in the algebraic case.)
\ms
(iii) The reduction of the decomposition theorem using a base change is trivial if $X$ is {\it smooth}. Indeed, assume there is a commutative diagram
$$\begin{array}{ccc}X&\buildrel{\,\pi}\over\longleftarrow&Y\\ \,\,\,\downarrow\!{\scriptstyle f}&&\,\,\,\downarrow\!{\scriptstyle g}\\ C&\buildrel{\,\,\pi'}\over\longleftarrow&D
\end{array}$$
where $X,Y$ are connected complex manifolds, $C,D$ are curves, $f,g$ are proper surjective morphisms, $\pi'$ is a finite morphism, and $\pi$ is a proper and generically finite \'etale morphism. Then the canonical morphism $A_{X}\to\Rb\pi_*A_{Y}$ {\it splits{\hskip1pt}} by composing it with its dual, using the self-duality of $A_X[d_X]$, $A_Y[d_Y]$ together with ${\rm Hom}_{D^b(X,A)}(A_X,A_X)=A$. Moreover, it is known that intersection complexes with local system coefficients are {\it stable{\hskip1pt}} under the direct images by {\it finite{\hskip1pt}} morphisms, see \cite{BBD}. So the decomposition theorem for $g$ implies that for $f$. (Here $C$ can be singular.)

\bigskip

\bigskip

\pagestyle{myheadings}

\end{document}